\documentclass[11pt,reqno]{article}
\usepackage{amsmath,amsfonts,amsthm}
\usepackage[margin=1in]{geometry}
\usepackage[normalem]{ulem}

\usepackage[hyphens]{url}
\usepackage{hyperref}
\usepackage{graphicx}
\usepackage{verbatim}
\usepackage{enumitem} % to control indention of "itemize" by [leftmargin=0.1in]
\usepackage{color}
\numberwithin{equation}{section}

\DeclareMathOperator*{\argmax}{arg\,max}

\newtheorem{theorem}{Theorem}[section]

\newtheorem{corollary}{Corollary}[section]
\newtheorem{lemma}{Lemma}[section]
\newtheorem{proposition}{Proposition}[section]

\newtheorem{definition}{Definition}[section]
\newtheorem{remark}{Remark}[section]
\newtheorem{example}{Example}[section]

\newcommand{\propref}{Proposition~\ref}

\renewcommand{\P}{\mathbb{P}}

\newcommand{\R}{\mathbb{R}}
\newcommand{\E}{\mathbb{E}}

\newcommand{\N}{\mathbb{N}}

\newcommand{\cQ}{\mathcal{Q}}

\newcommand{\A}{\mathcal{A}}

\newcommand{\eps}{\varepsilon}

\newcommand{\fP}{\mathfrak{P}}
\newcommand{\fT}{\mathfrak{T}}

\newcommand{\bq}{{\bf q}}

\newcommand{\nada}[1]{}
\definecolor{gb}{rgb}{0, 0.2, 0.8}

\title{Strong and Weak Equilibria for Time-Inconsistent Stochastic Control in Continuous Time}
\author{Yu-Jui Huang\thanks{
University of Colorado, Department of Applied Mathematics, Boulder, CO 80309-0526, USA, email: \texttt{yujui.huang@colorado.edu}. Partially supported by National Science Foundation (DMS-1715439) and the University of Colorado (11003573).}
 \and Zhou Zhou\thanks{
University of Sydney, School of Mathematics and Statistics, NSW 2006, Australia, email: \texttt{zhou.zhou@sydney.edu.au}.}
}
\date{\today}
\begin{document}
\maketitle

\begin{abstract}
A {\it new} definition of continuous-time equilibrium controls is introduced. %, for time-inconsistent stochastic control.  
As opposed to the standard definition, which involves a derivative-type operation, the new definition parallels how a discrete-time equilibrium is defined, and allows for unambiguous economic interpretation. The terms ``{\it strong equilibria}'' and ``{\it weak equilibria}'' are coined for controls under the new and the standard definitions, respectively. When the state process is a time-homogeneous continuous-time Markov chain, a careful asymptotic analysis gives complete characterizations of weak and strong equilibria. %: being weak is equivalent to dominance in the first-order term, while being strong amounts to a specific structure of higher-order terms. 
Thanks to Kakutani-Fan's fixed-point theorem, general existence of weak and strong equilibria is also established, under additional compactness assumption. Our theoretic results are applied to a two-state model under non-exponential discounting. In particular, we demonstrate explicitly that there can be incentive to deviate from a weak equilibrium, which justifies the need for strong equilibria. Our analysis also provides new results for the existence and characterization of discrete-time equilibria under infinite horizon. 
\end{abstract}

\textbf{MSC (2010):} 
%49K21, % Optimality conditions:	Problems involving relations other than differential equations
%60J05,  %	Discrete-time Markov processes on general state spaces
60J27,  %Continuous-time Markov processes on discrete state spaces
91A13, % Games with infinitely many players
93E20. % Optimal stochastic control
\smallskip


%%%%%%%%%%%%%%%%%%%%%%%%%%%%%%%%%%%%%%%%%%%%%%%%%%%%
%%%%%%%%%%%%%%%%%%%%%%%%%%%%%%%%%%%%%%%%%%%%%%%%%%%%

\section{Introduction}
Under time-inconsistency, an optimal rule derived today may not be optimal from the eyes of a future self. There is no ``dynamically optimal strategy'' that is good for the whole planning horizon, as opposed to standard time-consistent models.  A sensible reaction to time-inconsistency, introduced in Strotz \cite{Strotz55}, is to take future selves' behavior as a constraint, and find the best current action in response to that. %strategy that our future selves will actually follow over time. 
When every future self also reasons in this way, the resulting strategy is a (subgame perfect) equilibrium, from which no future self has incentive to deviate.  
 
This equilibrium approach, while widely-accepted, is highly nontrivial for stochastic control in continuous time. The upfront challenge is how to precisely define a continuous-time equilibrium. 

In discrete time, this is not a challenge at all:  Let $F(x,\alpha)$ be an objective function, depending on current state $x$ and the selected control $\alpha$. An equilibrium  $\alpha^*$ can be defined as
\begin{equation}\label{discrete def}
F(x,\alpha^*) \ge F(x,\alpha\otimes_1 \alpha^*),\quad \forall x\ \hbox{and}\ \alpha,
\end{equation}
where $\alpha\otimes_1 \alpha^*$ means that we apply $\alpha$ only at time $0$, and switch to $\alpha^*$ from time 1 on. The economic interpretation is clear: given that all future selves will follow $\alpha^*$, using any other control $\alpha$ at current time is no better than sticking to $\alpha^*$, i.e. no incentive for the current self to deviate from $\alpha^*$, conforming to the equilibrium idea. A continuous-time analogy to \eqref{discrete def} is far from obvious. Since the current self only exists at ``time point 0'', which carries no mass in continuous time, his decision to use a different strategy $\alpha$ normally has no effect on $F$. In other words, while one could replace the right hand side of \eqref{discrete def} by $\lim_{\eps\downarrow 0} F(x,\alpha\otimes_\eps \alpha^*)$, in most cases this limit equals $F(x,\alpha^*)$, leaving the comparison like \eqref{discrete def} meaningless.

Ekeland and Lazrak \cite{EL06} provided, for the first time, a precise definition of a continuous-time equilibrium: %Instead of comparing two objective values as in \eqref{discrete def}, they focused on the rate of change. 
roughly speaking, $\alpha^*$ is an equilibrium if  
\begin{equation}\label{continuous def}
\liminf_{\eps\downarrow 0} \frac{F(x,\alpha^*)-F(x,\alpha\otimes_\eps \alpha^*)}{\eps} \ge 0,\quad \forall x\ \hbox{and}\ \alpha.
\end{equation}
This formulation has spurred vibrant research on time-inconsistent control problems in continuous time, arising mainly in mathematical finance; see \cite{EP08}, \cite{EMP12}, \cite{HJZ12}, \cite{Yong12}, \cite{BMZ14}, and \cite{BKM17}, among many others. While \eqref{continuous def} has to some extent become the standard formulation of continuous-time equilibria, %its economic justification is in question. % 
it may {\it not} be fully justified in the economic sense.
%whether it can be {\it fully} justified economically is still in question. 
%The economic interpretation of \eqref{continuous def} is 

%Intuitively, what we desire from \eqref{continuous def} is $V(x,\alpha^*)\ge V(x,\alpha\otimes_\eps \alpha^*)$ as $\eps>0$ small enough, a kind of analogy to \eqref{discrete def}. This is, however, not guaranteed by \eqref{continuous def}. 
As pointed out in Bj\"{o}rk, Khapko, and Murgoci \cite[Remark 3.5]{BKM17}, \eqref{continuous def} does not correspond perfectly to the equilibrium concept: when \eqref{continuous def} holds with equality, $\alpha^*$ can be a stationary point that is not a maximum point. That is, it is possible that for some $x$ and $\alpha$, $F(x,\alpha^*)< F(x,\alpha\otimes_\eps \alpha^*)$ for all $\eps>0$, but the limit in \eqref{continuous def} is still zero. Then, when the current self is at the state $x$, there {\it is} incentive to deviate: following $\alpha$, in a however small interval $[0,\eps]$, is better than sticking to $\alpha^*$. In view of this,  $\alpha^*$ should not be considered as an equilibrium, {\it yet} it is included under \eqref{continuous def}. In short, \eqref{continuous def} may be too weak a definition to precisely reflect the equilibrium idea.

In this paper, a {\it new} definition of continuous-time equilibria is introduced: %generally speaking, 
$\alpha^*$ is an equilibrium if for any $x$ and $\alpha$, there exists $\eps^* =\eps^*(x,\alpha)>0$ such that 
\begin{equation}\label{continuous def'}
F(x,\alpha^*)\ge F(x,\alpha\otimes_\eps \alpha^*)\quad \hbox{for all $0<\eps<\eps^*$}. %,\quad \forall x\ \hbox{and}\ \alpha.
\end{equation}
This is analogous to \eqref{discrete def}, and admits the following %while the comparison is now taken under a sufficiently small horizon $\eps>0$, ...
economic interpretation: if \eqref{continuous def'} is violated for some $(x,\alpha)$, then for the current self at the state $x$, deviating to $\alpha$, in a however small interval $[0,\eps]$, is better than sticking to $\alpha^*$. Such incentive to deviate disappears when $\alpha^*$ is an equilibrium; see Remark~\ref{rem:interpretation} for details. Note that \eqref{continuous def'} entails \eqref{continuous def}, but not vice versa. Throughout this paper, we will call equilibria under \eqref{continuous def'} ``strong equilibria'' (Definition~\ref{def:strong E}), and those under \eqref{continuous def} ``weak equilibria'' (Definition~\ref{def:weak E}). The main goal of this paper is to elucidate the difference, as well as the connection, between strong and weak equilibria.\footnote{A notion similar to our ``strong equilibria'' appeared in He and Jiang \cite{HJ17} for a mean-risk portfolio selection problem. Yet, a detailed comparison between weak and strong equilibria was not their focus.} 

Specifically, we take the controlled state process $X$ to be a time-homogeneous continuous-time Markov chain. By selecting an appropriate generator $Q$ for $X$, an agent intends to maximize his expected cumulative running payoff over infinite horizon. The running payoff function is allowed to be time-dependent, making the problem time-inconsistent in general. This framework particularly covers optimization under non-exponential discounting. 

By detailed asymptotic analysis of the right hand side of \eqref{continuous def'}, now taking the form $F(x,Q\otimes_\eps Q^*)$, 
we establish complete characterizations of both weak and strong equilibria; see Theorems~\ref{thm:weak E iff} and \ref{thm:strong E iff}. In short, an equilibrium being weak amounts to dominance in the first-order term, while being strong demands a more delicate structure involving higher-order terms. This in turn leads to a handy machinery for finding weak and strong equilibria; see Propositions~\ref{prop:finding weak E}, \ref{coro:sufficient for strong E},~\ref{prop:sufficient for strong E}, and Theorem \ref{thm:strong E iff}. In a two-state model under pseudo-exponential discounting, such a machinery is demonstrated in detail through concrete examples. In particular, we construct explicitly a weak equilibrium $Q^*$ such that for some specific $(x,Q)$, $F(x,Q^*)< F(x,Q\otimes_\eps Q^*)$ for all $\eps>0$ small enough. That is, although $Q^*$ satisfies \eqref{continuous def}, there {\it is} incentive to deviate from $Q^*$ in a however small interval $[0,\eps]$, when the current state is $x$; see Example~\ref{eg:3} and Remark~\ref{rem:deviate weak E}. This justifies the need for the new notion of strong equilibria.

Note that the machinery for finding equilibria, while useful, is meant to be applied on a case-by-case basis, and does not say {\it a priori} whether an equilibrium exists. Thanks to Kakutani-Fan's fixed-point theorem, a general existence result for weak and strong equilibria can be established, under additional compactness assumption on the admissible set of $Q$; see Theorem~\ref{thm:existence}.

The literature of time-inconsistent stochastic control in continuous-time, as mentioned above, focuses solely on weak equilbria, which are usually characterized as (i) solutions to a system of nonlinear differential equations, the so-called {\it extended HJB system} (see e.g. \cite{EP08} and \cite{BKM17}), or (ii) the limit points of a sequence of discrete-time equilibria, when the discrete time mesh tends to zero (see e.g. \cite{Yong12} and \cite{BKM17}). Our analysis complements both (i) and (ii) above. 

First, note that (i) above is a {\it partial} characterization: if one finds smooth solutions to the nonlinear system, an equilibrium can be constructed from them. Yet, solving the system is generally difficult, and it is also not clear whether every equilibrium is related to such a system. By contrast, in our case where $X$ is a continuous-time controlled Markov chain, more tractable than a controlled diffusion mostly used in the literature, we obtain complete (i.e. ``if and only if'') characterization for weak equilibrium (Theorem~\ref{thm:weak E iff}), and an easy-to-check criterion for finding them (Proposition~\ref{prop:finding weak E}). On the other hand, we obtain in Theorem~\ref{thm:convergence} that discrete-time equilibria converge to weak equilibria in continuous time, under appropriate continuity assumption. Interestingly, it is guaranteed to converge only to weak, but not strong, equilibria. As shown in Example~\ref{eg:convergence}, a sequence of discrete-time equilibria converge uniquely to a weak equilibrium that is not strong. 

%It is worth pointing out that our analysis on weak equilibria {\it alone} readily complements the literature on time-inconsistency in both discrete and continuous time. 

Finally, %let us point out that 
our continuous-time analysis also sheds new light on discrete-time problems.  
In discrete time, an equilibrium is defined unambiguously as \eqref{discrete def}, and it can be found by straightforward backward sequential optimization in Pollak \cite{Pollak68}, when time horizon is finite. Under infinite horizon, such backward procedure breaks down; it is unclear whether an equilibrium exists in general, and a systematic way for finding equilibria is lacking. The continuous-time arguments in Theorem~\ref{thm:weak E iff} and Proposition~\ref{prop:finding weak E} turn out to be helpful: they can be modified to discrete time, giving a very general existence result for equilibria, as well as a handy criterion for finding equilibria, for the kind of time-inconsistent problems we focus on; see Theorem~\ref{thm:existence'} and Proposition~\ref{prop:finding E}.   

The paper is organized as follows. Section~\ref{sec:setup} introduces the setup of our time-inconsistent problem, and defines the two distinct notions of weak and strong equilibria. Section~\ref{sec:results} collects the main results, including complete characterization and general existence for both weak and strong equilibria. Section~\ref{sec:examples} applies the theoretic results to a concrete two-state model, %under non-exponential discounting, 
where we demonstrate explicitly that there can be incentive to deviate from a weak equilibrium. Section~\ref{sec:discrete} derives several new results for the corresponding discrete-time problem, and proves the convergence of discrete-time equilibria to a weak equilibrium. %, which need not be strong.  
Section~\ref{sec:conclusions} concludes the paper.

%%%%%%%%%%%%%%%%%%%%%%%
%%%%%%%%%%%%%%%%%%%%%%%

\section{The Setup and Definitions}\label{sec:setup}
Let $X=(X_t)_{t\geq 0}$ be a time-homogeneous continuous-time Markov chain taking values in $S:=\{1,2,\dotso,N\}$, for some $N\in\N$. The generator $Q\in\R^{N\times N}$ of $X$ is to be controlled. For each $i\in S$, we denote by $Q_i$ the $i^{th}$ row of the generator $Q$, and let
\begin{equation}\label{D_i}
D_i\subseteq E_i := \bigg\{q=(q_1,\dotso,q_N)\in\R^N:\ q_j\geq 0,\ j\neq i,\ q_i=-\sum_{j\neq i}q_j\bigg\}
\end{equation}
be %a closed convex set, which represents 
the admissible set of $Q_i$. The control space is then
\[
\cQ:=\big\{Q\in\R^{N\times N}:\ Q_i\in D_i,\ \forall i\in S\big\}.
\]
%where $Q_i$ represents the $i$-th row of $Q$. 

Consider a payoff function $f$ such that for any $t\ge 0$, $i\in S$, and $\bq\in D_i$, the value $f(t,i,\bq)\in \R$ stands for the payment rate at time $t$, given that $X_t=i$ and $Q_i=\bq$.
%$f: [0,\infty)\times S\times D_i\to\R$, with $f(t,i,\bq)$ standing for the payment rate at time $t$, given that $X_t=i$ and $\bq=Q_i$. 
We assume that 
\begin{equation}\label{f conti.}
\hbox{$f(\cdot,i,\bq)$ is continuous on $[0,\infty)$, for each $i\in S$ and $\bq\in D_i$.} 
\end{equation}
In addition, we impose the integrability condition
\begin{equation}\label{integrability}
\int_0^\infty\left(\sup_{i\in S,\ \bq\in D_i,\ \|\bq\|\le c}|f(t,i,\bq)|\right)dt<\infty\ \quad \forall c>0, 
\end{equation}
where $\|\cdot\|$ denotes the Euclidean norm in $\R^N$. 
Note that \eqref{f conti.} particularly implies that $t\mapsto \sup_{i\in S,\bq\in D_i}|f(t,i,\bq)|$ is lower semicontinuous, and thus Lebesgue measurable, which makes sense of the integration in \eqref{integrability}. For any $i\in S$ and $Q\in\cQ$, \eqref{integrability} guarantees that the expected payoff 
\begin{equation}\label{F}
F(i,Q):=\E_{i,Q}\left[\int_0^\infty f(t,X_t,Q_{X_t})dt\right]<\infty
\end{equation}
is well-defined, where $\E_{i,Q}$ denotes the expectation conditioned on $X_0 =i$ and the generator of $X$ being $Q$. Throughout this paper, we will write $\E_i$ for $\E_{i,Q}$ whenever there is no confusion about $Q$. 

In general, an agent who aims to maximize $F(i,Q)$ by selecting $Q\in Q$ may run into the issue of time-inconsistency. Specifically, an optimal control $Q^*\in \cQ$ for the problem 
\begin{equation}\label{v}
\sup_{Q\in\cQ} \E_{i}\left[\int_0^\infty f(t,X_t,Q_{X_t})dt\right]
\end{equation}
may depend on the initial state $i$, and we write it as $Q^*(i)$. At a later date $t>0$ with $X_t = j \neq i$, $Q^*(i)$ may no longer be optimal for the problem \eqref{v}, now with $i$ replaced by $j$, so that the agent is tempted to deviate to $Q^*(j)$, optimal in his view at time $t$. %different strategy $\tilde Q^*\in \cQ$ 

A typical example is optimization under non-exponential discounting. In this case, $f$ takes the form
\begin{equation}\label{NE discounting}
f(t,i,\bq) = \delta(t) g(i,\bq),
\end{equation}
where $\delta:[0,\infty)\to [0,1]$ is a discount function, assumed to be strictly decreasing with $\delta(0)=1$, and $g$ is a general measurable function. It is well-known that the problem \eqref{v} is time-consistent for the specific case $\delta(t):= e^{-\rho t}$ for some $\rho>0$, but time-inconsistent in general. 

\begin{remark}
The $t$ variable in $f(t,i,\bq)$ does not represent ``real calendar time'', but ``time difference'', i.e. the difference between the current time and the time of a future payoff. This is well-demonstrated in the discounting setup \eqref{NE discounting}. If $t$ in $f(t,i,\bq)$ were real calendar time, \eqref{F} would be time-inhomogeneous (i.e. $F(i,Q)$ should be $F(t,i,Q)$). This would make the problem \eqref{v} time-consistent, and thus not of interest for our studies. 
\end{remark}

As described in Strotz \cite{Strotz55}, when an agent is sophisticated enough to realize that his ``future selves'' will override his current plan (due to the lack of commitment), a sensible reaction is to take his future selves' behavior as a constraint, and choose the best present action in response to that. Assuming that all future selves reason in the same way, the agent searches for a (subgame perfect) equilibrium strategy, from which no future self has incentive to deviate.

While such equilibrium strategies have a straightforward definition in discrete time (see e.g. Definition~\ref{def:E} below), finding a precise continuous-time formulation had been a long-standing challenge. Ekeland and Lazrak \cite{EL06} provided, for the first time, a rigorous definition of a continuous-time equilibrium, using a derivative-type operation. This has spurred vibrant research on time-inconsistent stochastic control in continuous time, as mentioned in the introduction.

To formulate an equilibrium in the sense of \cite{EL06}, we introduce, for any $Q,Q'\in\cQ$ and $\eps>0$, the concatenation of $Q$ and $Q'$ at time $\eps$, denoted by $Q\otimes_\eps Q'$. Using this concatenated generator means that the evolution of $X$ is governed first by $Q$ on the interval $[0,\eps]$, and then by $Q'$ on $(\eps,\infty)$.

\begin{definition}\label{def:weak E}
$Q^*\in\cQ$ is called a weak equilibrium, if 
\begin{equation}\label{weak}
\liminf_{\eps\downarrow 0} \frac{F(i,Q^*)-F(i,Q\otimes_\eps Q^*)}{\eps} \ge 0,\quad \forall Q\in\cQ\ \hbox{and}\ i\in S.
\end{equation}
\end{definition}
Definition~\ref{def:weak E} involves a first-order inequality. % as in \eqref{weak}. 
This was introduced in \cite{EL06} as the definition of a continuous-time equilibrium, and followed by all subsequent research. Despite its popularity, this formulation may {\it not} be fully justified economically.
%follows the standard formulation of a continuous-time equilibrium,

Intuitively, what we desire from \eqref{weak} is $F(i,Q^*)\ge F(i,Q\otimes_\eps Q^*)$ as $\eps>0$ small enough, for all $(i,Q)$. %, such that one finds no incentive on the interval $[0,\eps]$ to deviate from $Q^*$, the future selves' strategy after time $\eps$. 
Yet, this is not ensured by \eqref{weak}. As pointed out in Bj\"{o}rk, Khapko, and Murgoci \cite[Remark 3.5]{BKM17}, the standard formulation, such as \eqref{weak}, does not correspond perfectly to the equilibrium concept: when $\eqref{weak}$ holds with equality, it is unclear whether $Q^*$ is a maximum point or a stationary point. In other words, it is possible that for some $Q\in\cQ$ and $i\in S$, $F(i,Q^*)< F(i,Q\otimes_\eps Q^*)$ for all $\eps>0$, but the limit in \eqref{weak} is still zero. Then, the agent at the state $i$ {\it does} have incentive to deviate: following $Q$, in a however small interval $[0,\eps]$, is better than sticking to $Q^*$. As such,  $Q^*$ should not be considered as an equilibrium, {\it yet} it is included under \eqref{weak}. %In short, \eqref{weak} is too weak a definition to precisely reflect the equilibrium idea.

This explains the terminology ``{weak} equilibrium'' in Definition~\ref{def:weak E}.  As opposed to that, we introduce the new notion of a {\it strong} equilibrium. 

\begin{definition}\label{def:strong E}
$Q^*\in\cQ$ is called a strong equilibrium if, for any $i\in S$ and $Q\in\cQ$ there exists $\eps>0$ such that
\begin{equation}\label{strong}
F(i,Q^*) \ge F(i,Q\otimes_{\eps'} Q^*)\quad \forall 0<\eps'\le\eps.
\end{equation}
\end{definition}

\begin{remark}\label{rem:interpretation}
Definition~\ref{def:strong E} admits the following economic interpretation. If \eqref{strong} is violated for some $(i,Q)$, then there exist $\{\eps_n\}_{n\in\N}$ such that $\eps_n\downarrow 0$ and $F(i,Q^*) < F(i,Q\otimes_{\eps_n} Q^*)$ for all $n\in\N$. Thus, for the agent at the state $i$, deviating to $Q$, in a however small interval $[0,\eps_n]$, $n\in\N$, is better than sticking to $Q^*$. Such incentive to deviate disappears when $Q^*$ is a strong equilibrium.
\end{remark}

It is of interest to investigate the relation between the standard notion of weak equilibria and our new concept of strong equilibria. Some immediate observations can be made.  

\begin{remark}\label{rem:observations}
By definition, a strong equilibrium is also a weak one. On the other hand, if a weak equilibrium satisfies \eqref{weak} with strict inequality for all $Q\in\cQ$ and $i\in S$, then \eqref{strong} must hold for any $i\in S$and $Q\in\cQ$, showing that the weak equilibrium is in fact strong. The unclear, challenging case is when $Q^*$ a weak equilibrium and 
%\begin{equation}\label{unclear}
\eqref{weak} holds with equality for some $Q\in\cQ$ and $i\in S$. 
%\end{equation}
\end{remark}

The goal of this paper is to elucidate the difference, as well as the connection, between strong and weak equilibria. This will be done at two different levels. Theoretically, complete characterizations for both weak and strong equilibria will be derived. Based on this, we will demonstrate how a weak equilibrium can differ from a strong one in concrete examples. In particular, we will show explicitly that there can be incentive to deviate from a weak equilibrium, as described in Remark~\ref{rem:interpretation}, which justifies the new notion of strong equilibria.

%%%%%%%%%%%%%%%%%%%%%%%%%%%%%
%%%%%%%%%%%%%%%%%%%%%%%%%%%%%

\section{The Main Results}\label{sec:results}

\subsection{Characterizations of Weak and Strong Equilibria}\label{subsec:characterizations}
In this section, we will carry out detailed asymptotic analysis of $F(i,Q^*) - F(i,Q\otimes_\eps Q^*)$ as $\eps\downarrow 0$.  This will %enable us to discern the precise difference between weak and strong equilibria, 
lead us to the distinct, yet connected, characterizations of weak and strong equilibria in Theorems~\ref{thm:weak E iff} and \ref{thm:strong E iff}. 

Recall $F(i,Q)$ in \eqref{F}. For any $i\in S$, $Q\in\cQ$, and $\eps>0$, we define
\begin{equation}\label{F_eps}
F_{\eps}(i,Q):=\E_i\left[\int_0^\infty f(t+\eps,X_t,Q_{X_t})dt\right].
\end{equation}
Then, we will write
\begin{align*}
F(Q)&:=(F(1,Q),\ \dotso,F(N,Q))\quad \hbox{and}\quad F_\eps(Q):=(F_\eps(1,Q),\ \dotso,F_\eps(N,Q)). %,\quad i\in S\ \hbox{and}\ Q\in\cQ.
\end{align*}
Also recall that $Q_i$ %$Q_i=(q_{i1},\ \dotso,q_{iN})$ 
denotes the $i^{th}$ row of $Q$. 

\begin{lemma}\label{lem:asymptotics}
Assume \eqref{f conti.} and \eqref{integrability}. Fix $i\in S$ and $Q, Q^*\in \cQ$. Then, as $\eps\downarrow 0$, 
\begin{equation}\label{expand F}
F(i,Q\otimes_{\eps} Q^*) = F_{\eps}(i,Q^*) + \left[f(0,i,Q_i)+F_\eps(Q^*)\cdot Q_i\right]\eps + o(\eps). 
%f(0,i,Q_i) \eps+F_{\eps}(i,Q^*)(1+q_{ii} \eps)+\sum_{j\neq i}F_{\eps}(j,Q^*) q_{ij} \eps + o(\eps).
\end{equation}
Suppose further that \eqref{f conti.} is strengthened to the following:
\begin{equation}\label{f conti.'}
|f(t+\eps,i,\bq)-f(t,i,\bq)| \le h(t,\eps;i,\bq)\quad \forall t\ge 0,\  \eps>0,\  i\in S,\  \hbox{and}\ \bq\in D_i,
\end{equation}
where $h$ is a nonnegative function such that %for each $t\ge 0$, $i\in S$ and $\bq\in D_i$,
\begin{align}
h(t,\eps;i,\bq)\ \hbox{is increasing in $\eps$},\ \hbox{with}\ \lim_{\eps\downarrow 0}h(t,\eps;i,\bq) =0;\label{h1}\\  
\int_0^\infty h(t,\eps;i,\bq) dt <\infty,\ \hbox{for}\ \eps>0\ \hbox{small enough}\label{h2}.
\end{align}
Then, as $\eps\downarrow 0$, 
\begin{equation}\label{expand F'}
F(i,Q^*)-F(i,Q\otimes_{\eps} Q^*) =  \left(\Gamma^{Q^*}(Q^*_i)- \Gamma^{Q^*}(Q_i)\right)\eps + o(\eps), 
\end{equation}
where
\begin{equation}\label{Gamma}
\Gamma^{Q^*}(Q_i) := f(0,i,Q_i)+Q_i\cdot F(Q^*).%,\quad \forall Q\in\cQ.
\end{equation}
\end{lemma}

The proof of Lemma~\ref{lem:asymptotics} is relegated to Appendix~\ref{sec:asymptotics}. 

\begin{remark}\label{rem:NE1}%[Non-exponential discounting]
Under non-exponential discounting as in \eqref{NE discounting}, assumptions in Lemma~\ref{lem:asymptotics}
%\eqref{f conti.}, \eqref{integrability}, and \eqref{f conti.'} all 
turn into mild conditions on the discount function $\delta$: \eqref{f conti.} amounts to the continuity of $\delta$; \eqref{integrability} reduces to 
\begin{equation}\label{delta integrable}
\int_0^\infty \delta(t) dt <\infty;
\end{equation}
\eqref{f conti.'} is equivalent to the continuity of $\delta$ and
\begin{equation}\label{delta integrable'}
\int_0^\infty \left(\delta(t)-\delta(t+\eps)\right) dt <\infty,\quad \hbox{for $\eps>0$ small enough}.
\end{equation}
Note that \eqref{delta integrable} implies \eqref{delta integrable'}, but not vice versa. Indeed, if \eqref{delta integrable} holds, the integral in \eqref{delta integrable'} reduces to $\int_0^\eps \delta(t)dt$, which is finite for all $\eps>0$. On the other hand, it can be checked that the hyperbolic discount function $\delta(t) := \frac{1}{1+\beta t}$, with $\beta>0$, satisfies \eqref{delta integrable'}, but not \eqref{delta integrable}.

Hence, ``\eqref{integrability} and \eqref{f conti.'}'' reduces to ``$\delta$ is continuous and satisfies \eqref{delta integrable}''.
This already covers many commonly-seen non-exponential discount functions, such as generalized hyperbolic $\delta(t) := \frac{1}{(1+\beta t)^k}$ with $\beta>0$ and $k>1$, and pseudo-exponential $\delta(t) := \lambda e^{-\rho_1 t} + (1-\lambda) e^{-\rho_2 t}$ with $\lambda\in (0,1)$ and $\rho_1,\rho_2>0$.   
\end{remark}

Suppose $X_0 =i$ and that all future selves beyond time $\eps>0$ will follow $Q^*\in\cQ$. In view of \eqref{expand F'}, the current self would like to follow $Q^*\in\cQ$ on $[0,\eps]$, rather than deviate to any other $Q\in\cQ$, only if the first-order term is always nonnegative,  i.e. for any $Q\in\cQ$,
\begin{equation}\label{optimal ROC}
\Gamma^{Q^*}(Q^*_i)\ge  \Gamma^{Q^*}(Q_i).
\end{equation}
%This is equivalent to, for all $Q\in\cQ$, 
%\begin{equation*}%\label{optimal ROC}
%f(0,i,Q_i)+F(Q^*)\cdot Q_i\leq f(0,i,Q_i^*)+F(Q^*)\cdot Q_i^*. %,\quad \forall Q\in\cQ.
%\end{equation*}
This relation completely characterizes weak equilibria.  

\begin{theorem}\label{thm:weak E iff}
Assume \eqref{integrability} and \eqref{f conti.'}. Then, $Q^*\in\cQ$ is a weak equilibrium if and only if $\Gamma^{Q^*}(Q^*_i) \ge \Gamma^{Q^*}(Q_i)$
%\eqref{optimal ROC} holds 
for all $(i,Q)\in S\times\cQ$. 
\end{theorem}

\begin{proof}
For any $i\in S$ and $Q\in \cQ$, by Lemma~\ref{lem:asymptotics} (specifically \eqref{expand F'}), as $\eps\downarrow 0$,
\begin{equation}\label{o(1)}
\frac{F(i,Q^*)-F(i,Q\otimes_{\eps} Q^*)}{\eps} = \left(\Gamma^{Q^*}(Q^*_i)- \Gamma^{Q^*}(Q_i)\right)+o(1).
\end{equation}
This shows that \eqref{weak} is satisfied (i.e. $Q^*$ is a weak equilibrium) if and only if \eqref{optimal ROC} holds for all $(i,Q)\in S\times\cQ$.
\end{proof}

Theorem~\ref{thm:weak E iff} gives rise to a handy criterion for weak equilibria. Specifically, for any $i\in S$ and differentiable function $v:D_i\to\R$, let $\nabla v(\alpha)$ be the gradient of $v$ evaluated at $\alpha\in D_i$, and $\partial_n v(\alpha)$ be the $n$-th component of $\nabla v(\alpha)$. Let
\begin{equation}\label{e26}
\fT:=\bigg\{\lambda=(\lambda_1,\dotso,\lambda_N)\in\R^N:\ \sum_{i=1,...,N}\lambda_i=0\bigg\}.
\end{equation}
%be the tangent vector space of $\fP$.

\begin{proposition}\label{prop:finding weak E}
Let $f(0,i,\cdot)$ be $\mathcal C_1$ for all $i\in S$. If $Q^*\in\cQ$ is a weak equilibrium, then for any $i\in S$,
\begin{equation}\label{e27}
\left(\nabla f(0,i,Q_i^*)+F(Q^*)\right)\cdot\lambda\leq 0,\quad\forall\lambda\in\fT\text{ s.t. }Q_i^*+\eps\lambda\in D_i\text{ for $\eps>0$ small enough}.
\end{equation}
In particular, if $Q_i^*$ is a relative interior point of $D_i$, then \eqref{e27} reduces to
\[
\partial_n f(0,i,Q_i^*)+F(n,Q^*)=\partial_m f(0,i,Q_i^*)+F(m,Q^*),\quad n,m=1,\dotso,N.
\]
Furthermore, if $f(0,i,\cdot)$ is additionally concave for all $i\in S$, the converse of \eqref{e27} is also true; that is, $Q^*$ is a weak equilibrium if and only if \eqref{e27} holds for all $i\in S$.
\end{proposition}

\begin{proof}
Let $Q^*$ be a weak equilibrium. Fix $i\in S$. For any $\eps>0$ and $\lambda\in\fT$ such that $Q^*_i+\eps\lambda\in D_i$, by Theorem~\ref{thm:weak E iff} and recalling \eqref{Gamma}, we get
$
f(0,i,Q^*_i)+ Q^*_i\cdot F(Q^*)\geq f(0,i,Q^*_i+\eps\lambda)+(Q^*_i+\eps\lambda)\cdot F(Q^*).
$
It follows that
\[
\frac{f(0,i,Q^*_i+\eps\lambda)-f(0,i,Q^*_i)}{\eps}+F(Q^*)\cdot\lambda\leq 0.
\]
Then \eqref{e27} follows as $\eps\downarrow 0$. Conversely, suppose $Q^*\in\cQ$ satisfies \eqref{e27} for all $i\in S$. If $f(0,i,\cdot)$ is concave for all $i\in S$, then the map
%\begin{equation}\label{e4'}
$\xi\mapsto f(0,i,\xi)+F(Q^*)\cdot\xi$
%\end{equation}
is concave for all $i\in S$. This, together with \eqref{e27}, shows that $\xi=Q^*_i$ is a global maximum of $\xi\mapsto f(0,i,\xi)+F(Q^*)\cdot\xi$ for all $i\in S$. That is, $\Gamma^{Q^*}(Q^*_i)\ge\Gamma^{Q^*}(Q_i)$ for all $(i,Q)\in S\times\cQ$. By Theorem~\ref{thm:weak E iff}, $Q^*$ is a weak equilibrium.
\end{proof}

The usefulness of Proposition~\ref{prop:finding weak E} will be apparent in Section~\ref{sec:examples}, where we look for equilibria in concrete examples.

%In general, a weak equilibrium $Q^*\in\cQ$ may satisfy \eqref{optimal ROC} with equality for some $i\in S$ and $Q\in \cQ\setminus \{Q^*\}$. In this case, the first-order term in \eqref{expand F'} (and thus in \eqref{o(1)}) is zero, leaving the argument in Corollary~\ref{coro:sufficient for strong E} inconclusive. To further examine if $Q^*\in\cQ$ is a strong equilibrium, one needs to upgrade \eqref{expand F'} to an expansion of second order or higher. 

To characterize strong equilibria, we need to upgrade \eqref{expand F'} to an expansion of second order or higher. To this end, whenever $f(\cdot,i,\bq)\in\mathcal{C}_1$, %and $f_t$ satisfying \eqref{integrability}, the function 
we define, for any $(i,Q)\in S\times\cQ$, the function
\[
G(i,Q) := \E_i\left[\int_0^\infty f_t(t,X_t,Q_{X_t})dt \right].
\]
In addition, we will write
\begin{align*}
G(Q)&:= \left(G(1,Q), G(2,Q),..., G(N,Q)\right),%\\ 
\qquad 
%\vec f(0,Q) &:= (f(0,1,Q_1), f(0,2,Q_2),...,f(0,N,Q_N)),\\
\Gamma^{Q^*}(Q) := \left(\Gamma^{Q^*}(Q_i), \Gamma^{Q^*}(Q_2), ..., \Gamma^{Q^*}(Q_N)\right).
\end{align*} 

%The next result upgrades \eqref{expand F'} to a second-order expansion.

\begin{lemma}\label{lem:asymptotics'}
Let $f$ satisfy \eqref{integrability} and $f(\cdot,i,\bq)$ be $\mathcal{C}_1$ on $[0,\infty)$ for all $i\in S$ and $\bq\in D_i$.
Assume additionally that $f_t$ also satisfies \eqref{integrability}, and that %the residual function 
\begin{equation}\label{r}
|f(t+\eps,i,\bq)-(f(t,i,\bq)+\eps f_t(t,i,\bq))|\le r(t,\eps;i,\bq)\quad t\ge 0,\ \eps>0,\ i\in S,\ \hbox{and}\ \bq\in D_i,
\end{equation}
where $r$ is a function continuous in $\eps$, and satisfies \eqref{h2} with % and there exists $\eps^*>0$ such that 
\begin{align}
%\exists \eps^*>0\ \hbox{such that for all $(t,i,\bq)$},\ 
\frac{r(t,\eps;i,\bq)}{\eps}\ \hbox{increasing in $\eps$},\quad \forall (t,i,\bq). %\quad \hbox{for some $\eps^*>0$}. 
\label{r1}%\ \forall (t,i,\bq). %\in [0,\infty)\times S\times 
\end{align}
Then, for any $i\in S$ and $Q$, $Q^*\in\cQ$, as $\eps\downarrow 0$,
\begin{align}
F(i,Q^*) - F(i,Q\otimes_\eps Q^*) &=  \left(\Gamma^{Q^*}(Q^*_i)- \Gamma^{Q^*}(Q_i)\right)\eps\notag\\
&\hspace{0.2in} +\frac12 \left(\Lambda^{Q^*}(i,Q^*)-\Lambda^{Q^*}(i,Q) \right)\eps^2 + o(\eps^2),\label{2rd expan.}
\end{align} 
where $\Gamma^{Q^*}(Q_i)$ is defined as in \eqref{Gamma} and
\begin{equation}\label{Lambda}
\Lambda^{Q^*}(i,Q):= %2 G(Q^*)\cdot Q_i + \vec f(0,Q)\cdot Q_i + f_t(0,i,Q_i)+F(Q^*)\cdot (Q^2)_i
f_t(0,i,Q_i)+ Q_i \cdot \left(2 G(Q^*) + \Gamma^{Q^*}(Q)\right). %,\quad \forall Q\in\cQ.  
\end{equation}
\end{lemma}

The proof of Lemma~\ref{lem:asymptotics'} is relegated to Appendix~\ref{sec:asymptotics'}. 

\begin{remark}
By Taylor's theorem, $f(\cdot,i,\bq)\in\mathcal{C}_1$ readily implies 
\begin{equation}\label{Taylor's}
r(t,\eps;i,\bq) = o(\eps),\quad \hbox{for each $(t,i,\bq)$}.
\end{equation}
Hence, there obviously exist a sequence $\{\eps_k\}_{k\in\N}$ with $\eps_k\downarrow 0$, depending on $(t,i,\bq)$, such that $r(t,\eps_k;i,\bq)/\eps_k$ decreases to 0. In view of this, \eqref{r1} is slightly stronger than ``$f(\cdot,i,\bq)\in\mathcal{C}_1$ for all $(i,\bq)$'': it requires $r(t,\eps_k;i,\bq)/\eps_k$ to decrease for any arbitrary $\{\eps_k\}_{k\in\N}$ with $\eps_k\downarrow 0$. %, independent of $(t,i,\bq)$.  
%A sufficient condition for \eqref{r1} is 
%\[
%f(\cdot,i,\bq)\in\mathcal{C}_2\quad \hbox{and}\quad \sup_{s\ge 0} |f_{tt} (s,i,\bq)| <\infty,\quad \hbox{for all}\ (i,\bq). 
%\] 
%Indeed, by Taylor's theorem, $f(\cdot,i,\bq)\in\mathcal{C}_2$ implies $r(t,\eps;i,\bq) = \frac12|f_{tt}(t^*,i,\bq)|\eps^2$ for some $t<t^*<t+\eps$. By the boundedness of $f_{tt}$, one may take $k(t,\eps;i,\bq) = \frac12 M_{i,\bq} \eps^2$, which satisfies \eqref{r1}.
\end{remark}

\begin{remark}\label{rem:NE2}
Under non-exponential discounting as in \eqref{NE discounting}, all conditions imposed in Lemma~\ref{lem:asymptotics'} boil down to mild conditions on the discount function $\delta$: 
\begin{itemize}
\item By Remark~\ref{rem:NE1}, ``$f$ satisfies \eqref{integrability}'' reduces to \eqref{delta integrable}. 
\item ``$f(\cdot,i,\bq)\in \mathcal{C}_1$ with $f_t$ satisfying \eqref{integrability}'' amounts to 
\begin{equation}\label{delta' integrable}
\delta\in \mathcal{C}_1\quad \hbox{and}\quad \int_0^\infty \delta'(t)dt <\infty;
\end{equation}
\item ``$r(t,\eps;i,\bq)$ satisfies \eqref{h2}'' reduces to $\int_0^\infty |\delta(t+\eps) -(\delta(t)+\eps \delta'(t))|dt <\infty$. Note that this is always true under \eqref{delta integrable} and \eqref{delta' integrable}; recall from Remark~\ref{rem:NE1} that \eqref{delta integrable} implies \eqref{delta integrable'}. 
\item ``$\frac{r(t,\eps;i,\bq)}{\eps}$ increasing in $\eps$, for all $(t,i,\bq)$'' boils down to 
\begin{equation}\label{delta' decrease}
\left|\frac{\delta(t+\eps)-\delta(t)}{\eps} - \delta'(t)\right|\quad \hbox{increasing in $\eps$},\quad \forall t\ge0. 
\end{equation}
A useful sufficient condition for \eqref{delta' decrease} is $\delta$ being convex.
\end{itemize}
Hence, conditions imposed in Lemma~\ref{lem:asymptotics'} reduce to \eqref{delta integrable}, \eqref{delta' integrable}, and \eqref{delta' decrease}. This already covers many commonly-seen non-exponential discount functions, including generalized hyperbolic and pseudo-exponential as mentioned in Remark~\ref{rem:NE1}. 
\end{remark}

The second-order expansion in Lemma~\ref{lem:asymptotics'} provides a straightforward sufficient condition for strong equilibria. It will be useful to show the existence of strong equilibria (Theorem~\ref{thm:existence} below), as well as to find strong equilibria explicitly in examples of Section~\ref{sec:examples}. Interestingly, the condition itself relates solely to the first-order term, yet the derivation of it involves the second-order term.

%{\color{red}We have a stronger version of Lemma 3.1, and the proof is essentially the same.\\
%Lemma: Assume there exists $a>0$ such that the function $h$ in (3.3) satisfies 
%$$\lim_\eps \sup_{0\leq y\leq a} h(t+y,\eps;i,{\bf q})=0\quad\text{and}\quad\int_0^\infty \sup_{0\leq y\leq a}h(t+y,\eps;i,{\bf q})dt<\infty.$$
%Note the assumption is satisfied if $h$ is non-increasing in $t$, e.g., decreasing impatience. Then for any $0\leq y\leq a$,
%$$\lim_\eps\left[\sup_{0\leq y\leq a}\left|\frac{1}{\eps}[F_y(i,Q^*)-F_y(i,Q\otimes_\eps Q^*)]-[\Gamma_y^{Q^*}(i,Q^*)-\Gamma_y^{Q^*}(i,Q)]\right|\right]=0,$$
%where
%$$\Gamma_y^{Q^*}(i,Q):=f(y,i,Q_i)+F_y(Q^*)\cdot Q_i.$$}

\begin{proposition}\label{coro:sufficient for strong E}
%Assume \eqref{integrability} and \eqref{f conti.'}. 
Suppose $f$ satisfies the conditions specified in Lemma~\ref{lem:asymptotics'}. For any $Q^*\in\cQ$, if 
\begin{equation}\label{strict condition}
\hbox{$\Gamma^{Q^*}(Q^*_i) > \Gamma^{Q^*}(Q_i)$ for all $i\in S$ and $Q\in \cQ$ with $Q_i\neq Q^*_i$},
\end{equation}
then $Q^*$ is a strong equilibrium.  
\end{proposition}

The proof of Proposition~\ref{lem:asymptotics'} is relegated to Appendix~\ref{sec:sufficient for strong E}. 

In general, a weak equilibrium $Q^*\in\cQ$ may satisfy \eqref{optimal ROC} with equality for some $i\in S$ and $Q\in \cQ$ with $Q_i\neq Q_i^*$. In this case, Proposition~\ref{coro:sufficient for strong E} is inconclusive. To further examine if $Q^*\in\cQ$ is a strong equilibrium, one needs to analyze the second-order term in \eqref{2rd expan.} carefully. 

%Lemma~\ref{lem:asymptotics'} gives rise to a useful criterion for strong equilibria.

\begin{proposition}\label{prop:sufficient for strong E}
Suppose $f$ satisfies the conditions specified in Lemma~\ref{lem:asymptotics'}. Let $Q^*\in Q$ be a weak equilibrium.  Consider 
\begin{equation}\label{R1}
R := \{(i,Q)\in S\times\cQ\setminus\{Q^*\}\ :\  \Gamma^{Q^*}(Q^*_i) = \Gamma^{Q^*}(Q_i)\}.
\end{equation}
If $\Lambda^{Q^*}(i,Q^*)>\Lambda^{Q^*}(i,Q)$ for all $(i,Q)\in R$, then $Q^*$ is a strong equilibrium. If $\Lambda^{Q^*}(i,Q^*)<\Lambda^{Q^*}(i,Q)$ for some $(i,Q)\in R$, then $Q^*$ is not a strong equilibrium.
\end{proposition}

\begin{proof}
Given $(i,Q)\in R^c$, Theorem~\ref{thm:weak E iff} implies $\Gamma^{Q^*}(Q^*_i) > \Gamma^{Q^*}(Q_i)$. Then, \eqref{expand F'} readily shows that $F(i,Q^*)-F(i,Q\otimes_{\eps} Q^*)> 0$ as $\eps>0$ small enough. Given $(i,Q)\in R$,  Lemma~\ref{lem:asymptotics'} implies  
\begin{equation}\label{o(1)'}
\frac{F(i,Q^*)-F(i,Q\otimes_{\eps} Q^*)}{\eps^2} = \frac12 \left(\Lambda^{Q^*}(i,Q^*)-\Lambda^{Q^*}(i,Q)\right)+o(1).
\end{equation}
If $\Lambda^{Q^*}(i,Q^*)>\Lambda^{Q^*}(i,Q)$ for all $(i,Q)\in R$, \eqref{o(1)'} shows that $F(i,Q^*)-F(i,Q\otimes_{\eps} Q^*)> 0$ as $\eps>0$ small enough, for all $(i,Q)\in R$. Thus, $Q^*$ is a strong equilibrium. If $\Lambda^{Q^*}(i,Q^*)<\Lambda^{Q^*}(i,Q)$ for some $(i,Q)\in R$,  \eqref{o(1)'} implies $F(i,Q^*)-F(i,Q\otimes_{\eps} Q^*)<0$ for $\eps>0$ small enough, showing that $Q^*$ is not a strong equilibrium.
\end{proof}

%Proposition~\ref{prop:sufficient for strong E} is an upgraded version of Corollary~\ref{coro:sufficient for strong E}, on the strength of the second-order expansion of $F(i,Q^*) - F(i,Q\otimes_\eps Q^*)$ in Lemma~\ref{lem:asymptotics'}. 

In view of Proposition~\ref{prop:sufficient for strong E}, there is the inconclusive case where $\Lambda^{Q^*}(i,Q^*)\ge \Lambda^{Q^*}(i,Q)$ for all $(i,Q)\in R$ and $\Lambda^{Q^*}(i,Q^*)=\Lambda^{Q^*}(i,Q)$ for some $(i,Q)\in R$. To resolve this, one needs to upgrade Proposition~\ref{prop:sufficient for strong E} further, with a higher-order expansion. Repeating this line of reasoning leads to the following characterization of strong equilibria. 

\begin{theorem}\label{thm:strong E iff}
Suppose there exist functions $L_n: S\times \cQ\times\cQ\to\R$, $n\in\N$, such that as $\eps\downarrow 0$,
\begin{equation}\label{ultimate}
F(i,Q^*) - F(i,Q\otimes_\eps Q^*) = \sum_{n=1}^\infty L_n(i,Q,Q^*) \eps^n,\quad \forall i\in S\ \hbox{and}\ Q,Q^*\in\cQ.
\end{equation}
Then, $Q^*\in\cQ$ is a strong equilibrium if and only if for any $(i,Q)\in S\times\cQ$, one of the following holds:
\begin{itemize}
\item [(i)]  $\exists\ n^*= n^*(i,Q)\in\N$ such that $L_n(i,Q,Q^*) =0$ for all $n<n^*$ and $L_{n^*}(i,Q,Q^*)>0$;
\item [(ii)] $L_n(i,Q,Q^*) =0$ for all $n\in \N$.
\end{itemize}
\end{theorem}

\begin{proof}
Let $Q^*\in\cQ$ be a strong equilibrium.  By contradiction, suppose there exists $(i,Q)\in S\times\cQ$ such that neither (i) nor (ii) holds. Then, there must exist $\hat n\in \N$ such that $L_n(i,Q,Q^*) =0$ for all $n<\hat n$ and $L_{\hat n}(i,Q,Q^*)<0$. Consequently, \eqref{ultimate} yields
\begin{equation}\label{eps^n}
\frac{F(i,Q^*) - F(i,Q\otimes_\eps Q^*)}{\eps^{\hat n}} = L_{\hat n}(i,Q,Q^*) +o(1).
\end{equation}
With $L_{\hat n}(i,Q,Q^*)<0$ , $F(i,Q^*) < F(i,Q\otimes_\eps Q^*)$ for $\eps>0$ small enough. This contradicts $Q^*$ being a strong equilibrium. 
On the other hand, suppose either (i) or (ii) holds for any $(i,Q)\in S\times\cQ$. If (i) holds, \eqref{ultimate} yields \eqref{eps^n}, with $\hat n$ replaced by $n^*$. With $L_{n^*}(i,Q,Q^*)>0$, $F(i,Q^*) > F(i,Q\otimes_\eps Q^*)$ for $\eps>0$ small enough. If (ii) holds, \eqref{ultimate} implies $F(i,Q^*) = F(i,Q\otimes_\eps Q^*)$ for $\eps>0$ small enough. This shows that $Q^*$ is a strong equilibrium. 
\end{proof}

Proposition~\ref{prop:finding weak E} and Theorem~\ref{thm:strong E iff} together provide a machinery for finding strong equilibria. First, one uses Proposition~\ref{prop:finding weak E} to find weak equilibria. Theorem~\ref{thm:strong E iff} comes into play next, when one wants to determine if a weak equilibrium $Q^*\in\cQ$ is in fact strong. In principle, one can derive higher-order expansions for $F(i,Q^*) - F(i,Q\otimes_\eps Q^*)$ to check whether (i) or (ii) in Theorem~\ref{thm:strong E iff} holds. Such derivations, in practice, can be quite technical and complicated, 
%: they require imposing sufficient regularity on $f$, which needs to be verified on a case-by-case basis.  
as shown in the proof of Lemma~\ref{lem:asymptotics'} (even for the second-order expansion). 

Since the main focus of this paper is to introduce and motivate the new notion of strong equilibria, we will not pursue expanding $F(i,Q^*) - F(i,Q\otimes_\eps Q^*)$ any further. 
%In the rest of the paper, we will not pursue expanding $F(i,Q^*) - F(i,Q\otimes_\eps Q^*)$ any further, 
As we will see in Section~\ref{sec:examples}, the second-order expansion in Lemma~\ref{lem:asymptotics'} already allows explicit demonstrations of how strong and weak equilibria can differ, and why the strong notion is needed.

%%%%%%%%%%%%%%%%%%%%%%%%

\subsection{General Existence of Equilibria under Compactness}

While one can use Proposition~\ref{prop:finding weak E} and Theorem~\ref{thm:strong E iff} to search for weak and strong equilibria, as discussed below Theorem~\ref{thm:strong E iff}, this machinery %is meant to be applied on a case-by-case basis, and 
does not assert {\it a priori} whether an equilibrium exists. It is therefore of interest to establish a general existence result for equilibria. This can be done by additional compactness assumption on admissible sets.

\begin{theorem}\label{thm:existence}
Suppose that $D_i$ in \eqref{D_i} is a convex compact set for all $i\in S$, and $f$ satisfies \eqref{integrability} with $f(0,i,\cdot)$ being concave for all $i\in S$. Then, there exists a weak equilibrium. If we assume additionally that $f(0,i,\cdot)$ is strictly concave for all $i\in S$ and $f$ satisfies the conditions in Lemma~\ref{lem:asymptotics'}, then there exists a strong equilibrium.
\end{theorem}

The proof of Theorem~\ref{thm:existence} is relegated to Appendix~\ref{sec:existence}.

\begin{remark}
Without compactness of $D_i$, the existence of an equilibrium, weak or strong, does not hold in general, even when $f(0,i,\cdot)$ is concave.
For instance, consider $S=\{1,2\}$, $f(t,1,\cdot)\equiv 0$, $f(t,2,\cdot)\equiv e^{-t}$, and $D_i = E_i$ for $i=1,2$, i.e. no constraint at all for the generator $Q = (q_{ij})_{i,j=1,2}$. %To maximize expected payoff, one intends to stay at state $2$ as long as possible. %which amounts to taking $q_{01}$ as large as possible and $q_{10}$ as small as possible. 
%This, however, precludes the existence of any equilibrium. Specifically, 
For any fixed $Q^*\in \cQ$, by the definition of $f$, we have $F(1,Q^*)<F(2,Q^*)$ and \eqref{optimal ROC} reads
\[
q_{12} \left(F(2,Q^*)-F(1,Q^*)\right) \le q^*_{12} \left(F(2,Q^*)-F(1,Q^*)\right)\quad \hbox{for}\ i=1. 
\]
This ineqaulity is violated as long as $q_{12}> q^*_{12}$. That is, \eqref{optimal ROC} does not hold for all $i\in S$ and $Q\in\cQ$, which precludes the existence of any weak (and thus strong) equilibrium, thanks to Theorem~\ref{thm:weak E iff}. 
\end{remark}

When each $D_i$ is convex and closed, but need not be bounded, we consider, for each $C>0$, the bounded set $D_i^C := \{\bq\in D_i: ||\bq||\le C\}$ and the corresponding set of generators
\[
\cQ_C:=\{Q\in\cQ: Q_i\in D_i^C,\ \forall i\in S\}. 
\]
Applying Theorem~\ref{thm:existence} to $\cQ_C$ gives the following result.

\begin{corollary}
Suppose $f$ satisfies \eqref{integrability} and $f(0,i,\cdot)$ is concave (resp. strictly concave) for all $i\in S$. For any $C>0$, there exists $Q^*_C\in \cQ_C$ such that \eqref{optimal ROC} holds for all $(i,Q)\in S\times\cQ_C$. Furthermore, if there is $C>0$ such that $\|(Q^{*}_C)_i\|<C$ for all $i\in S$,  then $Q_C^*$ is a weak (resp. strong) equilibrium.
%By \thmref{t3} the set of equilibriums $\mathcal{E}_C$ w.r.t. $\cQ_C:=\{Q\in\cQ:\ ||Q||\leq C\}$ is not empty. If there exists $C>0$ and $Q_C^*\in\mathcal{E}_C$ such that $||Q_C^*||<C$, then $Q_C^*$ is an equilibrium w.r.t. $\cQ$.
\end{corollary}

%%%%%%%%%%%%%%%%%%%%%%%%%%%%%%%%%%%%%%
%%%%%%%%%%%%%%%%%%%%%%%%%%%%%%%%%%%%%%

\section{A Two-State Model}\label{sec:examples}
In this section, we focus on a tractable two-state model under non-exponential discounting. Our goal is to demonstrate explicitly how theoretic results in Section~\ref{sec:results} can be used to find weak and strong equilibria, and how these two types of equilibria can differ from each other. 

Take $S=\{1,2\}$ and $D_i=E_i$ for $i=1,2$. Any generator $Q\in\cQ$ is then of the form %characterized by $(a,b)\in\R^2_+$, i.e.
\begin{equation}\label{Qab}
Q=
\left[\begin{matrix}
-a&a\\
b&-b
\end{matrix}\right],
\quad a,b\geq 0.
\end{equation}
We will denote this by $Q\sim(a,b)$. Consider the pseudo-exponential discount function
\begin{equation}\label{pedf}
\delta(t)=\lambda e^{-\rho t}+(1-\lambda) e^{-\rho' t}\quad t\ge 0,
\end{equation}
where $\lambda\in(0,1)$ and $\rho$, $\rho'\ge 0$ are given constants. Assume that 
\[
f(t,1,(-a,a))=\delta(t)g_1(a)\quad \hbox{and}\quad f(t,2,(b,-b))=\delta(t)g_2(b),
\]
for some given measurable functions $g_1$ and $g_2$. Given $Q\sim(a,b)$, we will write $F(i,Q)$ and $G(i,Q)$ as $F_i(a,b)$ and $G_i(a,b)$, respectively, for $i=1,2$. Observe that the transition probability of $X$, under $Q\sim(a,b)$, is given by
$$\left[
\begin{matrix}
P_{11}(t)&P_{12}(t)\\
P_{21}(t)&P_{22}(t)
\end{matrix}
\right]=
\left[
\begin{matrix}
\alpha+\beta e^{-\gamma t}&\beta-\beta e^{-\gamma t}\\
\alpha-\alpha e^{-\gamma t}&\beta+\alpha e^{-\gamma t}
\end{matrix}
\right],$$
where
$$\alpha:=\frac{b}{a+b},\quad\beta:=\frac{a}{a+b},\quad\gamma:=a+b.$$
Hence, we can calculate that for $i=1,2$,
\begin{align*}
F_i(a,b)&=\lambda F_i^\rho(a,b)+(1-\lambda) F_i^{\rho'}(a,b),\\
G_i(a,b)&=-\rho \lambda F_i^\rho(a,b)-\rho'(1-\lambda) F_i^{\rho'}(a,b),
\end{align*}
where %for $\phi=\rho,\rho'$,
\begin{align*}
F_1^\phi(a,b)& :=\left(\frac{\alpha}{\phi}+\frac{\beta}{\phi+\gamma}\right)g_1(a)+\left(\frac{\beta}{\phi}-\frac{\beta}{\phi+\gamma}\right)g_2(b),\quad \hbox{for}\ \phi=\rho,\rho'\\
F_2^\phi(a,b)& :=\left(\frac{\alpha}{\phi}-\frac{\alpha}{\phi+\gamma}\right)g_1(a)+\left(\frac{\beta}{\phi}+\frac{\alpha}{\phi+\gamma}\right)g_2(b),\quad \hbox{for}\ \phi=\rho,\rho'.
\end{align*}
Therefore,
\begin{align}
F_1(a,b)-F_2(a,b)&=\left(\frac{\lambda}{\rho+a+b}+\frac{1-\lambda}{\rho'+a+b}\right)(g_1(a)-g_2(b)),\label{e30}\\
G_1(a,b)-G_2(a,b)&=-\left(\frac{\rho\lambda}{\rho+a+b}+\frac{\rho'(1-\lambda)}{\rho'+a+b}\right)(g_1(a)-g_2(b)).\label{e30'}
\end{align}
It follows that for any $Q\sim (a,b)$ and $Q^*\sim(a^*,b^*)$,  $\Gamma^{Q^*}(Q_i)$ defined in \eqref{Gamma} takes the form
\begin{align}
\Gamma^{(a^*,b^*)}_1(a):= \Gamma^{Q^*}(Q_1) &= g_1(a)-a \left(F_1(a^*,b^*)-F_2(a^*,b^*)\right),\label{G1}\\
\Gamma^{(a^*,b^*)}_2(b):= \Gamma^{Q^*}(Q_2) &= g_2(b)+b \left(F_1(a^*,b^*)-F_2(a^*,b^*)\right).\label{G2}
%\\& = g_2(b)+b\left(\frac{\lambda}{\rho+a^*+b^*}+\frac{1-\lambda}{\rho'+a^*+b^*}\right)(g_1(a^*)-g_2(b^*)).
\end{align}
Moreover, $\Lambda^{Q^*}(i,Q)$ defined in \eqref{Lambda} takes the form
\begin{align}
&\Lambda^{(a^*,b^*)}_1(a,b):= \Lambda^{Q^*}(1,Q) = -2a \left(G_1(a^*,b^*)-G_2(a^*,b^*)\right)-a\left(g_1(a)-g_2(b)\right)\notag\\
&\hspace{1.8in} - (\rho\lambda+\rho' (1-\lambda)) g_1(a) + (a^2+ab) \left(F_1(a^*,b^*)-F_2(a^*,b^*)\right), \label{L1} \\
%& = a\left[2\left(\frac{\rho\lambda}{\rho+a^*+b^*}+\frac{\rho'(1-\lambda)}{\rho'+a^*+b^*}\right) + (a+b)\left(\frac{\lambda}{\rho+a^*+b^*}+\frac{1-\lambda}{\rho'+a^*+b^*}\right)\right](g_1(a^*)-g_2(b^*))\\
%&\hspace{0.2in} -a(g_1(a)-g_2(b))-(\rho\lambda+\rho'(1-\lambda))g_1(a).
&\Lambda^{(a^*,b^*)}_2(a,b):= \Lambda^{Q^*}(2,Q) = 2b \left(G_1(a^*,b^*)-G_2(a^*,b^*)\right)+b\left(g_1(a)-g_2(b)\right)\notag\\
&\hspace{1.8in} - (\rho\lambda+\rho' (1-\lambda)) g_2(b) - (b^2+ab) \left(F_1(a^*,b^*)-F_2(a^*,b^*)\right). \label{L2} 
\end{align}

The next example shows how \propref{prop:finding weak E} can be a convenient tool to find weak equilibria. Other results in Section~\ref{sec:results} can then be applied to check if a weak equilibrium is actually strong.

\begin{example}\label{eg:1}
Let $\lambda=\frac{1}{2},\ \rho=1,\ \rho'=2$, $g_1(a)=-a^2,\ g_2(b)=2-(1-b)^2$. By \propref{prop:finding weak E}, $Q\sim (a,b)$ is a weak equilibrium if and only if the following holds: (i) if $a,b>0$, we have
\begin{align}
\label{e28} g_1'(a)+F_2(a,b)-F_1(a,b)&=0,\\
\label{e29} g_2'(b)+F_1(a,b)-F_2(a,b)&=0,
\end{align}
and (ii) if $a=0$ (resp. $b=0$), then ``$\leq$'' holds in \eqref{e28} (resp. \eqref{e29}). Thanks to \eqref{e30}, the above equations admits a unique solution $(a^*,b^*)=(\frac{5}{12},\frac{7}{12})$. That is, $Q^*\sim (\frac{5}{12},\frac{7}{12})$ is the unique weak equilibrium. By Theorem~\ref{thm:weak E iff}, $a^*$ and $b^*$ are maximizers of $\Gamma^{(a^*,b^*)}_1(a)$ and $\Gamma^{(a^*,b^*)}_2(b)$, respectively; recall \eqref{G1} and \eqref{G2}. By the strict concavity of $g_1$ and $g_2$, $a^*$ and $b^*$ are in fact the unique maximizers. This shows that $Q^*\sim(\frac{5}{12},\frac{7}{12})$ is actually a strong equilibrium, thanks to Proposition~\ref{coro:sufficient for strong E}. 
\end{example}

An equilibrium can reside on the boundary of an admissible set, as the next example shows. 

\begin{example}\label{eg:2}
Let $\lambda=\frac{1}{2},\ \rho=1,\ \rho'=2$, $g_1(a)=-a^2,\ g_2(b)=2-b^2$. Using \propref{prop:finding weak E} as in Example~\ref{eg:1}, we obtain a unique weak equilibrium $Q^*\sim (a^*,0)$, where $a^*>0$ is the unique solution to
\[
-2a+\frac{1}{2}\left(\frac{1}{a+1}+\frac{1}{a+2}\right)(a^2+2)=0.
\]
By the strict concavity of $g_1$ and $g_2$, the same argument in Example~\ref{eg:1} shows that $Q^*\sim (a^*,0)$ is in fact a strong equilibrium. 
\end{example}

In the above two examples, weak equilibria are also strong, thanks to the strict concavity of $g_1$ and $g_2$. In general, a weak equilibrium may not be strong, and determining whether it is strong can be much more involved than applying Proposition~\ref{coro:sufficient for strong E}. This is demonstrated in the next example, where two equilibria co-exist: one is a weak equilibrium that is not strong; the other is strong.

\begin{example}\label{eg:3}
Let $\lambda=\frac{1}{2},\ \rho=1,\ \rho'=2$, $g_1(a)=-a^2$, and 
\begin{equation*}
g_2(b) =
\begin{cases}
 \frac{193}{144}+\frac56 b,\quad &\hbox{for}\ b< \frac{7}{12};\\
 2-(1-b)^2,\quad &\hbox{for}\ b\ge \frac{7}{12}. 
\end{cases} 
\end{equation*}
Note that $g_2$ is concave and $\mathcal{C}^1$ on $[0,\infty)$, but strictly concave only on $(\frac{7}{12},\infty)$. 

First, we claim that $(a^*,b^*)=(\frac{5}{12},\frac{7}{12})$ obtained in Example~\ref{eg:1} is still a weak equilibrium under current setting. Indeed, 
by \eqref{G1}, \eqref{G2} and \eqref{e30}, 
\begin{align*}
\Gamma^{(a^*,b^*)}_1(a) &=  g_1(a)- a\left(\frac{\lambda}{\rho+a^*+b^*}+\frac{1-\lambda}{\rho'+a^*+b^*}\right)(g_1(a^*)-g_2(b^*))= -a^2 + \frac56 a,\\
\Gamma^{(a^*,b^*)}_2(b) &=  g_2(b)+ b\left(\frac{\lambda}{\rho+a^*+b^*}+\frac{1-\lambda}{\rho'+a^*+b^*}\right)(g_1(a^*)-g_2(b^*))= g_2(b) - \frac56 b\\
& =
\begin{cases}
\frac{193}{144},\quad &\hbox{if}\ b< \frac{7}{12};\\
-\left(b^2-\frac{7}{12}\right) +\frac{193}{144},\quad &\hbox{if}\ b\ge \frac{7}{12}.
\end{cases},
\end{align*}
This shows that $\Gamma^{(a^*,b^*)}_1(a)$ is maximized uniquely at $a=a^*$, while 
\begin{equation}\label{G=G}
\argmax_{b\ge 0} \Gamma^{(a^*,b^*)}_2(b) = [0,{7}/{12}]. 
\end{equation}
By Theorem~\ref{thm:weak E iff}, this already implies that $Q^*\sim(a^*, b^*)$ is a weak equilibrium. 

As opposed to Examples~\ref{eg:1} and \ref{eg:2}, whether $Q^*$ is a strong equilibrium cannot be concluded by Proposition~\ref{coro:sufficient for strong E}, as $b^*=\frac{7}{12}$ is not a unique maximizer in \eqref{G=G}. We will instead resort to Proposition~\ref{prop:sufficient for strong E}. With the aid of \eqref{e30} and \eqref{e30'}, we deduce from \eqref{L2} that
\begin{align*}
&\Lambda^{(a^*,b^*)}_2(a^*,b)\\ 
&=b\left[-2\left(\frac{\rho\lambda}{\rho+a^*+b^*}+\frac{\rho'(1-\lambda)}{\rho'+a^*+b^*}\right) - (a^*+b)\left(\frac{\lambda}{\rho+a^*+b^*}+\frac{1-\lambda}{\rho'+a^*+b^*}\right)\right](g_1(a^*)-g_2(b^*))\\
&\hspace{0.2in} +b(g_1(a^*)-g_2(b))-(\rho\lambda+\rho'(1-\lambda))g_2(b) \\
&= b\left[\frac{7}{3}+\frac{5}{6}\left(b+\frac{5}{12}\right)\right] - \frac{25}{144} b - \left(b+\frac{3}{2}\right) g_2(b)=-\frac{1}{12}b-\frac{579}{288},\quad \hbox{for}\ b\le \frac{7}{12}.
\end{align*}
This shows that 
\begin{equation}\label{L<L}
\Lambda^{(a^*,b^*)}_2(a^*,b^*) < \Lambda^{(a^*,b^*)}_2(a^*,b),\quad \forall b\in [0,7/12).
\end{equation}
For any $Q\sim(a^*,b)$ with $b\in [0,7/12)$, \eqref{G=G} and \eqref{L<L} imply that $(2,Q)\in R$ (recall \eqref{R1}) and $\Lambda^{Q^*}(2,Q^*) < \Lambda^{Q^*}(2,Q)$. By Proposition~\ref{prop:sufficient for strong E}, $Q^*\sim (a^*,b^*)$ is not a strong equilibrium. 

Now, when using Proposition~\ref{prop:finding weak E} to find weak equilibria, if we take $b=0$, \eqref{e28} and \eqref{e29} become
\begin{align}
\frac56 &\le 2a = \frac12\left(\frac{1}{1+a}+\frac{1}{2+a}\right) \left(a^2+\frac{193}{144}\right). \label{e28'}
%\frac56 &\le - \frac12\left(\frac{1}{1+a}+\frac{1}{2+a}\right) \left(a^2+\frac{193}{144}\right) \label{e29'}.
\end{align}
This admits a unique solution $\bar a\in [0,\infty)$ (numerical computation shows $\bar a \approx 0.42364$). By \eqref{e28'},
\begin{align*}
\Gamma^{(\bar a,0)}_1(a) &=  -a^2+2\bar a a = -a (a-2\bar a),\quad \hbox{and}\quad \Gamma^{(\bar a,0)}_2(b) =   \frac{193}{144} + \left(\frac56-2\bar a\right) b. 
\end{align*}
This shows that $a=\bar a$ (resp. $b=0$) is the unique maximizer of $\Gamma^{(\bar a,0)}_1(a)$ (resp. $\Gamma^{(\bar a,0)}_2(b)$). Hence, $\bar Q^*\sim(\bar a, 0)$ is a strong equilibrium, thanks to Proposition~\ref{coro:sufficient for strong E}. 
\end{example}

\begin{remark}\label{rem:deviate weak E}
In the above example, for any $Q\sim(a^*,b)$ with $b\in [0,7/12)$, we deduce from \eqref{G=G}, \eqref{L<L}, and the second-order expansion in Lemma~\ref{lem:asymptotics'} that 
\[
F(2,Q^*) < F(2,Q\otimes_\eps Q^*),\quad \hbox{for all $\eps>0$ small enough}. 
\]
This shows that, although $Q^*\sim (a^*,b^*)$ is a weak equilibrium, there is incentive to deviate from $Q^*$ at state 2: deviating to $Q$, in a however small interval $[0,\eps]$, yields a larger payoff than sticking to $Q^*$. This reminds us of Remark~\ref{rem:interpretation}, and indicates the need for the notion of strong equilibria. 
\end{remark}

%%%%%%%%%%%%%%%%%%%%%%%%%

\subsection{Application to Machinery Management}
A machine is any mechanical or electrical device that converts input energy to useful output energy or work. In the good state, where the machine functions properly, there is a tradeoff between achieving maximal efficiency and reducing tear and wear. While one intends to exert input energy intensely enough to maximize the payoff generated by the machine, more intense use of the machine will bring about the bad state, where the machine is out of order, more easily. In the bad state, one spends effort repairing the machine. There is again a tradeoff: the more intensely the effort is spent, the faster the good state can be restored; yet, at the same time, the faster costs accumulate.  All these considerations of machinery management can be well encoded in our two-state model. 

Take Example~\ref{eg:1} for instance. Let $i=1$ be the bad state and $i=2$ be the good state. For any $Q\sim (a,b)$, $a\ge 0$ represents how intensely effort is spent on repairing the machine (in state 1), and $b\ge 0$ stands for how intensely the machine is used (in state 2). In view of \eqref{Qab}, the larger $a\ge 0$ (i.e. the more intensely effort is spent), the faster (on average) the machine will function again (i.e. the state will switch from $1$ to $2$). The payoff function $g_1(a) = -a^2$, however, shows that the cost of repair grows quickly with the intensity of effort. Similarly, in view of \eqref{Qab}, the smaller $b\ge 0$ (i.e. the less intensely the machine is used), the less likely the machine will break down (i.e. the state will switch from $2$ to $1$).  Leaving the machine idle (i.e. $b=0$), however, may not be the best choice in view of the payoff function $g_2(b)=2-(1-b)^2$. As the intensity of input energy increases (i.e. $b\ge 0$ increases), the instantaneous payoff $g_2(b)$ first increases to its maximal level $g_2(1)=2$ at $b=1$, and then decreases indefinitely. Thus, one may want to choose $b\ge 0$ closer to 1 to possibly enlarge the cumulative payoff. 

In a factory or a company, how a machine should be managed, i.e. how $Q\sim (a,b)$ should be specified, is often decided by a group of professional workers, instead of one single individual. In such group decision making, pseudo-exponential discounting is typically used. Its most basic form is \eqref{pedf}, which is the discount function for a group that involves two individuals (or cohorts) who discount exponentially at different rates $\rho$ and $\rho'$, respectively, with $\lambda\in(0,1)$ determined by the sizes or influence of the cohorts. In other words, the derivation in Example~\ref{eg:1} amounts to finding a management plan for a machine---how hard it should be used and repaired---that will be consistently carried out over time. It turns out that only $Q^*\sim (\frac{5}{12},\frac{7}{12})$ is such a time-consistent management plan.

%%%%%%%%%%%%%%%%%%%%%%%%%%%%%%%%%%%%%%%%%%%%%%%%%%%%%%%
%%%%%%%%%%%%%%%%%%%%%%%%%%%%%%%%%%%%%%%%%%%%%%%%%%%%%%%

\section{The Discrete-Time Case}\label{sec:discrete}
In this section, we study the discrete-time model corresponding to that in Section~\ref{sec:setup}. The purpose is twofold. First, when time horizon is infinite, little is known about the existence and characterization of equilibria, even in discrete time. Arguments from Section~\ref{sec:results} can be applied here to shed new light on this. Our second focus is the convergence of discrete-time equilibria to their continuous-time counterparts, as the mesh size in time diminishes. As we will see, when discrete-time equilibria converge, they always converge to a weak equilibrium, which, however, need not be strong. 

Let $X=(X_t)_{t=0,1,...}$ be a time-homogeneous discrete-time Markov chain taking values in $S:=\{1,2,\dotso,N\}$ for some $N\in\N$, and take $\bar \N:=\{0,1,2,...\}$. The transition matrix $u=(u_{ij})_{i,j=1,\dotso,N}$ of $X$ is to be controlled. Let $\mathfrak{P}$ be the set of probability measures defined on $S$, i.e.
\begin{equation}\label{tp}
\mathfrak{P}:=\bigg\{\alpha=(\alpha_1,\dotso,\alpha_N)\in\R_+^N:\  \sum_{i=1,...,N}\alpha_i=1\bigg\}.
\end{equation}
Consider a continuous function $\kappa: \bar\N\times S\times\mathfrak{P}\to\R$. For any $(t,i,\alpha)\in \bar\N\times S\times\fP$, $\kappa(t,i,\alpha)$ represents the payoff at time $t$, given that $X_t=i$ and $u_i=\alpha$, where $u_i$ denotes the $i^{th}$ row of $u$. Assume that
\begin{equation}\label{e2}
\sum_{t=0}^\infty\bigg(\sup_{(i,\alpha)\in S\times\fP}\left|\kappa(t,i,\alpha)\right|\bigg)<\infty.
\end{equation}
Let $\A_i\subseteq \mathfrak P$ be the set of admissible transitional probabilities when $X$ is at the state $i$. %We assume that $\A_i$ is closed and convex. 
Define 
\[
\A := \{u\in R^{N\times N} : u_i\in \A_i,\ \forall i\in S\}. 
\]
For any $i\in S$ and $u\in\A$, \eqref{e2} guarantees that the expected payoff 
\begin{equation}\label{e0}
V(i,u):=\E_{i,u}\bigg[\sum_{t=0}^\infty \kappa(t,X_t,u_{X_t})\bigg]<\infty
\end{equation}
is well-defined, where $\E_{i,u}$ denotes the expectation conditioned on $X_0 =i$ and the transition matrix of $X$ being $u$. We will write $\E_i$ for $\E_{i,u}$ whenever there is no confusion about $u$. 

For any $u,u^*\in\A$, we introduce the concatenation of $u$ and $u^*$ at time $1$, denoted by $u\otimes_1 u^*$. Using this concatenated matrix means that the evolution of $X$ is governed by $u$ at time $0$, and then by $u^*$ at all subsequent time points. Given an initial state $i\in S$, the expected value is then
\begin{equation}\label{e1}
V(i,u\otimes_1 u^*) =\kappa(0,i,u_i)+\sum_{j=1}^N\left(\E_{j,u^*}\left[\sum_{t=0}^\infty \kappa(t+1,X_t,u^*_{X_t})\right]\cdot u_{ij}\right).
\end{equation}

\begin{definition}\label{def:E}
$u^*\in\A$ is called an equilibrium if,  for any $(i,u)\in S\times\A$, $V(i,u^*)\ge V(i,u\otimes_1 u^*)$. 
\end{definition}

\begin{remark}\label{rem:parallels}
A strong equilibrium (Definition~\ref{def:strong E}) parallels the above discrete-time definition, and admits a clear economic interpretation, as explained in Remark~\ref{rem:interpretation}. By contrast, the precise interpretation of a weak equilibrium (Definition~\ref{def:weak E}) is not so clear in the literature. 
\end{remark}

A handy characterization of equilibria can be established, by following the arguments in Proposition~\ref{prop:finding weak E}. To this end,  for any $i\in S$ and $u\in\A$, define
\begin{equation}\label{H}
H_i(u):=\E_{i,u}\left[\sum_{t=0}^\infty \kappa(t+1,X_t,u_{X_t})\right]\quad\text{and}\quad H(u):=(H_1(u),\dotso,H_N(u)).
\end{equation}
For any differentiable function $v:\, \fP\mapsto\R$, let $\nabla v(\alpha)$ be the gradient of $v$ evaluated at $\alpha\in\fP$, and $\partial_n v(\alpha)$ be the $n^{th}$ component of $\nabla v(\alpha)$. Also recall $\fT$ in \eqref{e26}. Proposition~\ref{prop:finding E} below can be proved by following line by line the proof of Proposition~\ref{prop:finding weak E}.

\begin{proposition}\label{prop:finding E}
Let $\kappa(0,i,\cdot)$ be $\mathcal{C}_1$ for all $i\in S$. % (the differentiability is extended to the boundary of $\fP$ in a natural way). 
If $u^*\in\A$ be an equilibrium, then for any $i\in S$,
\begin{equation}\label{e3}
\left(\nabla \kappa(0,i,u^*_i)+H(u^*)\right)\cdot\lambda\leq 0,\quad\forall\lambda\in\fT\text{ s.t. }u^*_i+\eps\lambda\in\A_i\text{ for $\eps>0$ small enough}.
\end{equation}
In particular, if $u^*_i$ is a relative interior point of $\A_i$, then
$$\partial_n \kappa(0,i,u^*_i)+H_n(u^*)=\partial_m \kappa(0,i,u^*_i)+H_m(u^*),\quad n,m=1,\dotso,N.$$
Furthermore, if $\kappa(0,i,\cdot)$ is additionally concave for any $i\in S$, then the converse is also true; that is, $u^*$ is an equilibrium if and only if \eqref{e3} holds.
\end{proposition}

To establish the existence of equilibria, arguments used to prove Theorem~\ref{thm:existence} can also be applied here. Since $\mathfrak{P}$ in \eqref{tp} is by definition compact, we no longer need the compactness assumption in Theorem~\ref{thm:existence}, leading to the following very general existence result. 

\begin{theorem}\label{thm:existence'}
Suppose $\A_i$ is convex and closed, and $\kappa(0,i,\cdot)$ is concave, for all $i\in S$. Then, there exists an equilibrium.
\end{theorem}
The proof of Theorem~\ref{thm:existence'} is relegated to Appendix~\ref{sec:existence'}.

\begin{remark}
If $\kappa(0,i,\cdot)$ is not concave, then in general an equilibrium may not exist. For instance, take $S=\{1,2\}$ and denote any transition matrix 
\begin{equation}\label{tM}
u=
\begin{bmatrix}
    1-\alpha  & \alpha \\
    \beta      & 1-\beta
    \end{bmatrix},\quad \alpha,\beta\in[0,1]
\end{equation}
by $u\sim (\alpha,\beta)$. Let $\delta(0)=1$ and $\delta(t)=k\cdot\lambda^{t-1}$ for $t\in \N$, where $k\in (0,1)$ is a constant. Consider
\[
\kappa(t,1,u_1)=\delta(t)g_1(\alpha)\quad\text{and}\quad \kappa(t,2,u_1)=\delta(t)g_2(\beta),
\]
where 
\[
g_1(\alpha):=-\frac{7}{8}\sqrt{\alpha}\quad\text{and}\quad g_2(\beta):=2-\sqrt{1-\beta}, 
\]
which are strictly convex. %We will show that an equilibrium does not exist. 
Given $u^*\sim (\alpha^*,\beta^*)$ and $u\sim (\alpha,\beta)$, direct calculation shows
\begin{align*}
V_1(\alpha) := V(1,u\otimes_1 u^*)&=g_1(\alpha)+H_1(u^*)\cdot(1-\alpha)+H_2(u^*)\cdot \alpha,\\%\label{1111}\\
V_2(\beta) := V(2,u\otimes_1 u^*)& =g_2(\beta)+H_1(u^*)\cdot \beta+H_2(u^*)\cdot(1-\beta).%\label{2222}.
\end{align*}
If $u^*\sim (\alpha^*,\beta^*)$ is an equilibrium, $V_1(\alpha)$ and $V_2(\beta)$ attain maximums at $\alpha=\alpha^*$ and $\beta=\beta^*$, respectively. The strict convexity of $g_1$ and $g_2$ implies $\alpha^*, \beta^*\in \{0,1\}$. If $u^*\sim (0,0)$, then $H_1(u^*)=0$ and $H_2(u^*)=\frac{3}{2}$, and thus
$V_1(\alpha) =-\frac{7}{8}\sqrt{\alpha}+\frac{3}{2}\alpha.$ But, $V_1(0)=0<V_1(1)=\frac{5}{8}$ implies $(0,0)$ cannot be an equilibrium. Similar calculation shows none of $(0,1)$, $(1,0)$, $(1,1)$ is an equilibrium.
\end{remark}

%%%%%%%%%%%%%%%%%%%%

\subsection{Convergence to Continuous Time}

Recall the continuous-time setup in Section~\ref{sec:setup}. For the continuous-time payoff function $f$, we further assume that there exists $T>0$, independent of $i$ and $\bq$, such that  
\begin{equation}\label{e33}
t\mapsto |f(t,i,\bq)|\quad \hbox{is nonincreasing}, \quad \hbox{for $t\ge T$}. 
\end{equation}
Take $(\delta_n)_{n\in\N}$ in $\R$ with $\delta_n\downarrow 0$. For each $n\in\N$, define $\kappa^n: \bar\N\times S\times\mathfrak{P}\to \R$ by
\begin{equation}\label{f^n}
\kappa^n(k,i,\alpha):=f(k\delta_n,i,\tilde\alpha^{i,n})\cdot\delta_n,
\end{equation}
where 
\[
\tilde\alpha^{i,n}:=\frac{1}{\delta_n}(\alpha_1,\dotso,\alpha_{i-1},\alpha_i-1,\alpha_{i+1},\dotso,\alpha_N),
\]
for each $\alpha=(\alpha_1,\dotso,\alpha_N)\in\mathfrak{P}$. For any transition matrix $u$, define the generator $Q^{u,n}=(q_{ij}^{u,n})$ by
\begin{equation}\label{Q^u,n}%e36
q_{ij}^{u,n}:=
\begin{cases}
\frac{1}{\delta_n}u_{ij},& j\neq i,\\
\frac{1}{\delta_n}(u_{ii}-1),& j=i.
\end{cases}
\end{equation}
Then, we introduce 
\begin{equation}\label{A^n}
\mathcal{A}^n:=\{u\ \text{transition matrix}: Q^{u,n}\in\mathcal{Q}\}.
\end{equation}

For each $n\in\N$, consider the discretized problem $V^n$ given by \eqref{e0}, with $\kappa$ and $\A$ replaced by $\kappa^n$ and $\mathcal{A}^n$. 
Suppose that there exists an equilibrium $u^n\in\mathcal{A}^n$, i.e. $V^n(i,u^n)\ge V^n(i,u\otimes_1 u^n)$ for all $i\in S$ and $u\in\A^n$. Following the notation in \eqref{H}, this means for any $i\in S$ and $u\in\mathcal{A}^n$,
\begin{equation}\label{e35}
\kappa^n(0,i,u^n_i)+H^n(u^n)\cdot u^n_i\geq \kappa^n(0,i,u_i)+H^n(u^n)\cdot u_i,
\end{equation}
where $H^n$ is defined as in \eqref{H} with $\kappa$ replaced by $\kappa^n$, i.e.
\begin{equation}\label{e31}
H_i^n(u^n):=\E_{i,u^n}\left[\sum_{k=0}^\infty \kappa^n(k+1,X_k,u^n_{X_k})\right],\quad \forall i\in S.
\end{equation}
In the following, we will write
\[
Q^n := Q^{u^n,n}\in \cQ.
\]
for simplicity. 
%$F^n(u^n)=(F_1^n(u^n),\dotso,F_N^n(u^n))$ with
%\begin{equation}\label{e31}
%F_i^n(u):=\E_{i,u^n}\left[\sum_{k=0}^\infty f^n(k+1,Y_k,u^n(Y_k))\right].
%\end{equation}
%Denote $(q_{ij}^n):=Q^n:=Q^{u^n,n}$ for simplicity. 
The main convergence result is the following.

\begin{theorem}\label{thm:convergence}
Assume \eqref{integrability}, \eqref{e33}, and that $f(\cdot,i,\cdot)$ is continuous for all $i\in S$. 
If there exists $Q^*\in\mathcal{Q}$ such that (up to a subsequence) $Q^n\to Q^*$, then $Q^*$ satisfies \eqref{optimal ROC} for all $(i,Q)\in S\times \cQ$. That is, $Q^*$ is a weak equilibrium (under Definition~\ref{def:weak E}). %for the problem in continuous time. 
\end{theorem}

\begin{remark}
Suppose that for the continuous-time problem $F$ in \eqref{F}, $D_i$ is a closed convex set and $f(0,i,\cdot)$ is concave for all $i\in S$. Then, in view of \eqref{f^n} and \eqref{A^n}, Theorem~\ref{thm:existence'} implies that an equilibrium $u^n\in\A^n$ exists for the discretized problem $V^n$, for all $n\in\N$. If we further assume that $D_i$ is bounded for all $i\in S$, then $(Q^n)_{n\in\N}$ is pre-compact. Then, by Theorem~\ref{thm:convergence}, any limit point of $(Q^n)_{n\in\N}$ is a weak equilibrium (under Definition~\ref{def:weak E}).
\end{remark}

To establish Theorem~\ref{thm:convergence}, we need the following technical lemma.

\begin{lemma}\label{l1}
Assume \eqref{integrability}, \eqref{e33}, and that $f(\cdot,i,\cdot)$ is continuous for all $i\in S$. 
If there exists $Q^*\in\mathcal{Q}$ such that (up to a subsequence) $Q^n\to Q^*$, then for any $i\in S$,
%Recall $F_i^n(u^n)$ defined in \eqref{e31}. We have that
\[
H_i^n(u^n)\rightarrow F_i(Q^*)=\E_{i,Q^*}\left[\int_0^\infty f(t,X_t,Q^*_{X_t}) dt\right]\quad \hbox{as}\ n\to\infty.
\]
\end{lemma}

The proof of Lemma~\ref{l1} is relegated to Appendix~\ref{sec:l1}. 

\begin{proof}[Proof of Theorem \ref{thm:convergence}]
For any $i\in S$ and $u\in\mathcal{A}^n$, \eqref{e35} can be re-written as 
\begin{align*}
f(0,i,Q^n_i)&\cdot\delta_n+\sum_{j\neq i}H_j^n(u^n)\cdot u_{ij}^n+H_i^n(u^n)(u_{ii}^n-1)\\
&\geq f(0,i,Q^{u,n}_i)\cdot\delta_n+\sum_{j\neq i}H_j^n(u^n)\cdot u_{ij}+H_i^n(u^n)(u_{ii}-1).
\end{align*}
%Hence, for any $i\in S$ and $u\in\mathcal{A}^n$,
Dividing both sides by $\delta_n$ yields
\[
f(0,i,Q^n_i)+\sum_{j\neq i}H_j^n(u^n)\cdot q_{ij}^n+H_i^n(u^n)\cdot q_{ii}^n\geq f(0,i,Q^{u,n}_i)+\sum_{j\neq i}H_j^n(u^n)\cdot q_{ij}^{u,n}+H_i^n(u^n)\cdot q_{ii}^{u,n},
\]
which is equivalent to %, we have that for any $i\in S$ and $Q\in\mathcal{A}$,
$
f(0,i,Q^n_i)+H^n(u^n)\cdot Q^n_i\geq f(0,i,Q_i)+H^n(u^n)\cdot Q_i.
$
Thanks to the continuity of $f(0,i,\cdot)$ and Lemma \ref{l1}, sending $n\rightarrow\infty$ gives \eqref{optimal ROC}.
\end{proof}

In general, the limit point $Q^*$ in Theorem~\ref{thm:convergence} need not be a strong equilibrium. This is demonstrated in the next example: the equilibria for discretized problems converge uniquely to a weak equilibrium for the continuous-time problem, but this weak equilibrium is not strong.

\begin{example}\label{eg:convergence}
Consider the two-state model in Section~\ref{sec:examples}. Take $\lambda=\frac{1}{2}$, $\rho=1$, $\rho'=2$, 
$$g_2(b)\equiv 0\quad\text{and}\quad g_1(a)=-\frac{1}{4}a^4+ka^3-k^2a^2-\frac{3}{4}a-1,$$
where $k>0$ is a constant. Note that $g_1'(a)=-a(a-k)(a-2k)-\frac{3}{4}$ and $g_1(0)=-1.$

First, we show that $Q^*\sim(0,0)$ is a weak equilibrium that is not strong. By \eqref{e30},  $F_1(0,0)-F_2(0,0)=\frac{3}{4}g_1(0)=-\frac{3}{4}.$
Then, for any $Q\sim (a,b)$, \eqref{G1} and \eqref{G2} imply
$$\Gamma^{Q^*}(Q_1)=g_1(a)+\frac{3}{4}a\quad\text{and}\quad\Gamma^{Q^*}(Q_2)=-\frac{3}{4}b.$$
This shows that $\Gamma^{Q^*}(Q_1)$ (resp. $\Gamma^{Q^*}(Q_2)$) is maximized at $a=0$ and $a=2k$  (resp. at $b=0$). Hence, Theorem~\ref{thm:weak E iff} readily implies $Q^*\sim(0,0)$ is a weak equilibrium. On the other hand, consider $\hat Q\sim(2k,0)$. Note that $\Gamma^{Q^*}(Q^*_1)=\Gamma^{Q^*}(\hat Q_1)$, and direct calculation shows
\[
\Lambda^{Q^*}(1,Q^*)=\frac{3}{2}<\frac{33}{4}k+\frac{3}{2}=\Lambda^{Q^*}(1,\hat Q),
\]
thanks to \eqref{e30'} and \eqref{L1}. Thus, by Proposition~\ref{prop:sufficient for strong E}, $Q^*\sim(0,0)$ is not a strong equilibrium. 

%For $\eps>0$ small enough we will show that
%\begin{equation}\label{e201}
%F_1(\hat Q\otimes_\eps Q^*)>F_1(Q^*).
%\end{equation}
%Since
%$$\Gamma^{Q^*}(1,Q^*)=\Gamma^{Q^*}(1,\hat Q),$$
%we will use the second order criterion. In fact, we can compute and get that
%$$\Lambda^{Q^*}(1,Q^*)=\frac{3}{2}<\frac{33}{4}k+\frac{3}{2}=\Lambda^{Q^*}(1,\hat Q).$$
%Then \eqref{e201} follows and thus $Q^*$ is not a strong equilibrium.

Now, consider the discretized problems. Denote by $u\sim (\alpha, \beta)$ the transition matrix given as in \eqref{tM}. %Similarly to Section~\ref{sec:examples}, we denote any transition matrix 
%$$u=
%\begin{bmatrix}
%1-\alpha&\alpha\\
%\beta&1-\beta
%\end{bmatrix},\quad \alpha,\beta \in [0,1]
%$$
%by $u\sim(\alpha,\beta)$. 
For each $h>0$, consider the discretized problem $V^h$ as in \eqref{e0}, where $\kappa$ is replaced by %and the payoff functions (after divided by time step $h>0$) can be written as
$$\kappa^h(n,1,u_1)=\delta(nh)g_1(\alpha/h) h\quad\text{and}\quad \kappa^h (n,2,u_2)=\delta(nh)g_2(\beta/h) h\equiv 0,$$
with $\delta(\cdot)$ as in \eqref{pedf}. We claim that $u^*\sim(0,0)$ is an equilibrium for $V^h$, for $h>0$ small enough.

Fix $u\sim(\alpha,\beta)\neq(0,0)$. Recall \eqref{e0} and \eqref{e1}. If $\beta>0$, since $g_1(0)<0$, we have
\begin{equation}\label{V2}
V^h(2,u\otimes_1 u^*)=\beta\sum_{n=0}^\infty \delta((n+1) h)g_1(0) h<0=V^h(2,u^*),\quad \forall h>0.
\end{equation}
If $\alpha>0$, then using $g_1(0)=-1$,
\begin{align*}
V^h(1,u\otimes_1 u^*)&=g_1(\alpha/h) h +(1-\alpha)\sum_{n=1}^\infty\frac{1}{2}(e^{-nh}+e^{-2nh})g_1(0) h\\
&= h \left[ g_1(\alpha/h)+\frac{\alpha}{2}\sum_{n=1}^\infty(e^{-nh}+e^{-2nh})-\frac{1}{2}\sum_{n=1}^\infty(e^{-nh}+e^{-2nh})\right].
\end{align*}
It can be checked that
$$\sum_{n=1}^\infty e^{-nh}=\frac{1}{h}\left(1-\frac{1}{2}h+o(h)\right),$$
implying
$$V^h(1,u\otimes_1 u^*)=h \bigg[\tilde g(h)-\frac{1}{2}\sum_{n=1}^\infty(e^{-nh}+e^{-2nh})\bigg],$$
where
$\tilde g(h):=g_1(\alpha/h)+\frac{\alpha}{h}\left(\frac{3}{4}-\frac{1}{2}h+o(h)\right).$
For $h>0$ small enough, $\tilde g(h)<g_1(\alpha/h)+\frac{3}{4}\frac{\alpha}{h}\le g_1(0)=-1,$ where the last inequality follows from the fact that $a\mapsto g_1(a)+\frac{3}{4}a$ is maximized at $a=0$ and $a=2k$ (this was mentioned above when we maximized $\Gamma^{Q^*}(Q_1)$). This yields 
\begin{equation}\label{V1}
V^h(1,u\otimes_1 u^*)<\frac{1}{2}\sum_{n=0}^\infty(e^{-nh}+e^{-2nh})g_1(0) h=V^h(1,u^*).
\end{equation}
By \eqref{V2} and \eqref{V1}, $u^*\sim (0,0)$ is an equilibrium for $V^h$, for $h>0$ small enough. In view of \eqref{Q^u,n}, $Q^h := Q^{u^*,h} \to Q^* \sim (0,0)$ holds trivially, as $Q^h\sim(0,0)$ for $h>0$ small enough. 
\end{example}

It seems somewhat surprising that discrete-time equilibria need not converge to a strong equilibrium. After all, as observed in Remark~\ref{rem:parallels}, strong equilibria {\it conceptually} parallel discrete-time equilibria: Definitions~\ref{def:strong E} and \ref{def:E} both require {\it direct} dominance in value, instead of {\it indirect} dominance via the rate of change in value (as stipulated  for weak equilibria in Definition~\ref{def:weak E}). 

%Intriguingly, 
Despite this conceptual resemblance, 
achieving ``direct dominance in value'' in continuous time is much more demanding {\it technically} than that in discrete time.
%, in order to achieve ``direct dominance in value'', strong equilibria {have to} be {\it technically} distinct from discrete-time equilibria, despite their conceptual resemblance. 
For discretized problems, the direct dominance ``$V^n(i,u^n)\ge V^n(i,u\otimes_1 u^n)$'' is simply \eqref{e35}. For the continuous-time problem, the direct dominance ``$F(i,Q^*)\ge F(i,Q\otimes_\eps Q^*)$ for $\eps>0$ small enough'' amounts to nonnegativity of  \eqref{2rd expan.}. Crucially, as $n\to\infty$, \eqref{e35} gives nonnegativity of {\it only} the first-order term of \eqref{2rd expan.} (i.e. \eqref{optimal ROC}), instead of the entire \eqref{2rd expan.}. That is to say, the very concept ``direct dominance in value'' is much harder to achieve in continuous time, thereby posing a technical gap between discrete-time equilibria and strong equilibria. 

Appropriate conditions on the payoff function $f$ can bridge the gap. 

\begin{remark}
If ${\bf q}\mapsto f(t,i,{\bf q})$ is strictly concave, then the limit point $Q^*$ in Theorem~\ref{thm:convergence} is a strong equilibrium, thanks to Proposition \ref{coro:sufficient for strong E}.
\end{remark}

We demonstrate in the next example that with ${\bf q}\mapsto f(t,i,{\bf q})$ being strictly concave, the equilibria for discretized problems indeed converge to a strong equilibrium. % for the continuous-time problem.

\begin{example}
Recall the setting in Example \ref{eg:1}. Take $(\delta_n)_{n\in\N}$ in $\R$ with $\delta_n\downarrow 0$, and define $\kappa^n$ and $H_i^n$ as in \eqref{f^n} and \eqref{e31}. Given  $u\sim (\alpha,\beta)$ as in \eqref{tM}, we abuse the notation slightly by writing
$\kappa^n(k,1,\alpha)=\delta(k\delta_n)\cdot g_1(\alpha/\delta_n)\cdot\delta_n$, $\kappa^n(k,2, \beta)=\delta(k\delta_n)\cdot g_2(\beta/\delta_n)\cdot\delta_n$, and
$$H_i^n(\alpha,\beta):=\E_{i,u}\bigg[\sum_{k=0}^\infty\kappa^n(k+1,X_k, u_{X_k})\bigg].$$
As $\kappa^n(0,i,\cdot)$ is concave, by Proposition \ref{prop:finding E}, $u^n\sim (\alpha,\beta)$ is an equilibrium (with respect to $\kappa^n$) if and only if $(\alpha,\beta)$ satisfies
\begin{equation}\label{ea1}
(\kappa^n(0,1,\alpha))'+H_2^n(\alpha,\beta)-H_1^n(\alpha,\beta)
\begin{cases}
\leq 0,& \alpha=0,\\
=0,& \alpha\in (0,1),\\
\geq 0,& \alpha=1,
\end{cases}
\end{equation}
\begin{equation}\label{ea2}
(\kappa^n(0,2,\beta))'+H_1^n(\alpha,\beta)-H_2^n(\alpha,\beta)
\begin{cases}
\leq 0,& \beta=0,\\
=0,& \beta\in (0,1),\\
\geq 0,& \beta=1,
\end{cases}
\end{equation}
where the derivatives are taken with respect to $\alpha$ and $\beta$, respectively. Direct calculation yields
$$H_1^n(\alpha,\beta)-H_2^n(\alpha,\beta)=(g_1(\alpha/\delta_n)-g_2(\beta/\delta_n))\cdot \delta_n \cdot\sum_{k=0}^\infty\delta((k+1)\delta_n)\cdot (1-\alpha-\beta)^k.$$
%where $\gamma(\alpha,\beta):=1-\alpha-\beta$. 
Then, for $\delta_n>0$ small enough, the solution $(\alpha_n,\beta_n)$ to \eqref{ea1} and \eqref{ea2} is given by
$$\alpha_n:=\frac{\delta_n^2}{2}\left(\frac{e^{-\delta_n}}{{1-e^{-\delta_n}(1-\delta_n)}}+\frac{e^{-2\delta_n}}{{1-e^{-2\delta_n}(1-\delta_n)}}\right),\quad\beta_n:=\delta_n-\alpha_n. $$
In view of \eqref{Q^u,n}, each discrete-time equilibrium $u^n\sim (\alpha_n,\beta_n)$ gives rise to the generator $Q^n \sim \big(\frac{\alpha_n}{\delta_n},\frac{\beta_n}{\delta_n}\big) $.  As $n\to\infty$, it can be checked that $Q^n\sim\big(\frac{\alpha_n}{\delta_n},\frac{\beta_n}{\delta_n}\big)$ converges to $Q^*\sim \left(\frac{5}{12},\frac{7}{12}\right)$, which is a strong equilibrium as shown in Example \ref{eg:1}.
\end{example}

%%%%%%%%%%%%%%%%%%%%%%%%%%%%%%%
%%%%%%%%%%%%%%%%%%%%%%%%%%%%%%%

\section{Conclusions}\label{sec:conclusions}

In this paper, we introduce the new notion of {\it strong equilibria}, as a refinement of the standard formulation of continuous-time equilibria  (which we call {\it weak equilibria}). As we have shown, there are situations where one finds it beneficial to deviate from a weak equilibrium, indicating that this standard formulation does not correspond perfectly to the equilibrium idea. A strong equilibrium, by contrast, is defined analogously to a discrete-time equilibrium, and admits an unambiguous interpretation of no deviation. 
To elucidate the difference and connection between these two types of equilibria, we assume that the state process is a continuous-time finite-state Markov chain on an infinite horizon. This allows us to derive complete characterizations of strong and weak equilibria, and compare them explicitly in concrete examples. It is of interest to investigate the case where the state process is a diffusion process, the typical setup in prior literature on time-inconsistent stochastic control in continuous time. A recent working paper He and Jiang \cite{HJ18} pursues this direction. In a diffusion model, they observe from several classical time-inconsistent problems that strong equilibria seem quite elusive. They propose ``{\it regular} equilibria'', a new class of equilibria that are slightly weaker than the strong ones, and show that regular equilibria are more tractable. Still, the existence of strong equilibria in a diffusion model demands further exploration. 

In fact, whenever an infinite state space or a finite time horizon is considered, the {existence} of equilibria is a genuine issue. Note that the existence proof in Section \ref{sec:existence} requires the set of admissible controls, i.e. $\cQ$, to be compact. In this paper, since $\cQ$ is finite-dimensional (as a subset of $\R^{N\times N}$), its compactness can be easily checked and appropriately assumed through closedness and boundedness. An infinite state space or a finite time horizon, however, renders the set of admissible controls infinite-dimensional. With no straightforward characterization of compactness in an infinite-dimensional space, the existence of equilibria is largely obscure. %available, it is hard to use Kakutani-Fan's fixed-point theorem to show the existence of equilibria as in Section A.4.

Time-inconsistent stopping problems in continuous time, on the other hand, have not received much attention until very recently. Interestingly, two distinct formulations of equilibrium stopping rules emerge from recent developments, and they, in some spirit, correspond to weak and strong equilibria in the control case. In \cite{EWZ2018} and \cite{CL17}, the derivative-type operation as in \eqref{weak} is followed closely to define an equilibrium stopping rule, which corresponds to a weak equilibrium in our case. On the other hand, \cite{HN18} and \cite{HNZ17} define an equilibrium by comparing the value of sticking to future selves' strategy and the value of deviating to another strategy at current time. This is similar to the comparison in \eqref{discrete def}, and thus in principle closer to a strong equilibrium in our case. It is of interest to investigate the precise relation between these two types of equilibrium stopping rules, as what we have done here for strong and weak equilibria.

%%%%%%%%%%%%%%%%%%%%%%%%%%%%%%%
%%%%%%%%%%%%%%%%%%%%%%%%%%%%%%%

\appendix

\section{Proofs for Section~\ref{sec:results}}

\subsection{Proof of Lemma~\ref{lem:asymptotics}}\label{sec:asymptotics}
Since $X$ evolves according to $Q$ on the time interval $[0,\eps]$, $\P(X_\eps =i\mid X_0=i) = 1+ q_{ii} \eps + o(\eps)$ and $\P(X_t=j\mid X_0=i) = q_{ij} \eps + o(\eps)$. This, together with \eqref{integrability}, implies that
\begin{align}\label{first}
F(i,Q\otimes_{\eps} Q^*) &=  \E_i\left[\int_0^\eps f(t,X_t,Q_{X_t})dt\right] + F_{\eps}(i,Q^*)(1+q_{ii} \eps)+\sum_{j\neq i}F_{\eps}(j,Q^*) q_{ij} \eps + o(\eps)\nonumber\\
&= F_{\eps}(i,Q^*) + \E_i\left[\int_0^\eps f(t,X_t,Q_{X_t})dt\right] + (F_\eps(Q^*)\cdot Q_i)\eps +o(\eps). 
\end{align}
Consider 
\begin{equation}\label{tau}
\tau:= \inf\{t\ge 0: X_t\neq X_0\}.
\end{equation}
Given that $X_0=i$, recall that $\tau$ is exponentially distributed with parameter $\lambda = -q_{ii}$, and thus
\begin{equation}\label{q_ii}
\lim_{\eps\downarrow 0} \frac{\P(\tau\le \eps\mid X_0 =i)}{\eps} = \lim_{\eps\downarrow 0} \frac{1-e^{q_{ii} \eps}}{\eps} = -q_{ii}. 
\end{equation}
Now, observe that 
\begin{align}
\E_i\left[1_{\{\tau > \eps\}} \int_0^\eps f(t,X_t,Q_{X_t})dt\right] &= \E_i\left[1_{\{\tau > \eps\}} \int_0^\eps f(t,i,Q_{i})dt\right] = \P(\tau> \eps\mid X_0 =i) \int_0^\eps f(t,i,Q_{i})dt\notag\\
&= (1+q_{ii}\eps + o(\eps)) f(s^*,i,Q_i) \eps,\quad \hbox{for some $0<s^* \le  \eps$}\notag\\
&= (1+q_{ii}\eps + o(\eps)) (f(0,i,Q_i) \eps + o(\eps)) = f(0,i,Q_i) \eps + o(\eps),\label{>eps}
\end{align}
where the second line is due to \eqref{q_ii} and the continuity of $t\mapsto f(t,i,Q_i)$ on $[0,\eps]$, and the third line is from the continuity of $t\mapsto f(t,i,Q_i)$ at $0$ from the right. Let $c:= \sup_{i\in S} \|Q_i\|<\infty$. Then, %by the boundness of $f$ in \eqref{pre-e32}, there exists $M>0$, independent of $i\in S$, $Q\in\cQ$, and $\eps>0$, such that
\begin{equation}\label{<eps}
\left|\E_i\left[1_{\{\tau \le \eps\}} \int_0^\eps f(t,X_t,Q_{X_t})dt\right]\right|\le 
\int_0^\eps \left(\sup_{i\in S,\ \bq\in D_i,\ \|\bq\|\le c} |f(t,i,\bq)|\right) dt\cdot \P(\tau\le \eps\mid X_0 =i).
%\P(\tau\le \eps\mid X_0 =i)^{1/2} \E_i\left[ \left(\int_0^\eps f(t,X_t,Q_{X_t})dt\right)^2\right]^{1/2}
%M  \P(\tau\le \eps\mid X_0 =i) = o(\eps), 
\end{equation}
Thanks to \eqref{integrability}, the Lebesgue integral on the right hand side of \eqref{<eps} is finite and converges to 0 as $\eps \to 0$. It then follows from \eqref{q_ii} that the right hand side of \eqref{<eps} is of $o(\eps)$. Combining this and \eqref{>eps}, we obtain from \eqref{first} the desired result \eqref{expand F}. 

Now, taking $Q= Q^*$ in \eqref{expand F} gives
%\begin{equation*}%\label{expand F}
$F(i,Q^*) = F_{\eps}(i,Q^*) + \left[f(0,i,Q^*_i) + F_\eps(Q^*)\cdot Q_i^*\right] \eps + o(\eps).$ 
%\end{equation*}
This, together with \eqref{expand F}, yields
\[
F(i,Q^*)-F(i,Q\otimes_\eps Q^*)=([f(0,i,Q_i^*)+F_\eps(Q^*)\cdot Q_i^*]-[f(0,i,Q_i)+F_\eps(Q^*)\cdot Q_i])\eps+o(\eps).
\]
Hence, to prove \eqref{expand F'}, it remains to show that $F_\eps(i,Q)=F(i,Q)+o(1)$ for all $i\in S$. 
For each $t\ge 0$ and $\eps>0$, define
$$H(t,\eps):=\sum_{i\in S}h(t,\eps,i,Q_i)<\infty,$$
where the finiteness follows from $S$ being a finite set. Under \eqref{f conti.'}, we have
\[
|F_\eps(i,Q)-F(i,Q)|\leq\E_i\left[\int_0^\infty h(t,\eps;X_t,Q_{X_t})dt\right] \le \int_0^\infty H(t,\eps)dt\to 0,\quad \hbox{as $\eps\downarrow 0$}.
\]
where the convergence comes from the dominated convergence theorem and \eqref{h1}. Note that the dominated convergence theorem is applicable here thanks to \eqref{h1} and \eqref{h2}.

%%%%%%%%%%%%%%%%%%%%%%%%

\subsection{Proof of Lemma~\ref{lem:asymptotics'}}\label{sec:asymptotics'}

Since $F(i,Q\otimes_\eps Q^*) = \E_i\left[\int_0^\eps f(t,X_t, Q_{X_t}) dt\right] +\E_i\left[\int_\eps^\infty f(t,X_t, Q^*_{X_t}) dt\right]$, we will deal with the two terms on the right hand side one-by-one. 

To handle $\E\left[\int_0^\eps f(t,X_t, Q_{X_t}) dt\right]$, consider the events $A$, $B$, and $C$ that on the interval $[0,\eps]$, the state of $X$ does not change, changes exactly once, and changes twice or more, respectively. Take $\tau$ in \eqref{tau} and recall that it is exponentially distributed with parameter $-q_{ii}$. Thus,
%\begin{align*}
$\P(A) = \P(\tau>\eps\mid X_0=i) = e^{q_{ii}\eps} = 1+ q_{ii}\eps + \frac12 q^2_{ii} \eps^2 + o(\eps^2).$ %\\
%\P(B) &= 1- \P(A)-\P(C) =  -q_{ii}\eps - \frac12 q^2_{ii} \eps^2 + o(\eps^2).
%\end{align*}
It follows that
\begin{align}
&\E_i\left[\int_0^\eps f(t,X_t,Q_{X_t}) dt\ \middle|\ A\right] \P(A) = \P(A) \int_0^\eps f(t,i,Q_{i}) dt  \nonumber\\
& = \left(1+ q_{ii}\eps + \frac12 q^2_{ii} \eps^2 + o(\eps^2)\right) \left(\int_0^\eps \left( f(0,i,Q_i) + t f_t(0,i,Q_i)\right) dt +o(\eps^2)\right)\label{Tt}\\
& = \left(1+ q_{ii}\eps + \frac12 q^2_{ii} \eps^2 + o(\eps^2)\right) \left( f(0,i,Q_i) \eps + \frac12 f_t(0,i,Q_i)\eps^2 + o(\eps^2)\right)\notag\\
&=  f(0,i,Q_i) \eps + \left(q_{ii} f(0,i,Q_i) + \frac12 f_t(0,i,Q_i)\right) \eps^2 +o(\eps^2).\label{A part}
\end{align}
Here, \eqref{Tt} follows from the estimate 
\begin{equation}\label{for Tt}
\left|\int_0^\eps f(t,i,Q_i) dt -\int_0^\eps \left(f(0,i,Q_i)+t f_t(0,i,Q_i)\right)dt \right| \le \int_0^\eps r(0,t;i,Q_i)dt = \eps r(0,t(\eps);i,Q_i),
\end{equation}
for some $0<t(\eps)< \eps$. By \eqref{Taylor's}, the last term above is $o(\eps^2)$. 
On the other hand, let $\eta$ be the density function of $\tau$, given that $\tau\le \eps$. Since $\P(\tau\le \ell \mid\tau\le \eps) = \frac{\P(\tau\le\ell)}{\P(\tau\le\eps)} = \frac{1-e^{
q_{ii}\ell}}{1-e^{q_{ii}\eps}}$ for all $\ell\in(0,\eps]$, 
\begin{equation}\label{eta}
\eta(\ell) = \frac{d}{d\ell}\left(\frac{1-e^{q_{ii}\ell}}{1-e^{q_{ii}\eps}}\right) = \frac{-q_{ii}e^{q_{ii}\ell}}{1-e^{q_{ii}\eps}},\quad \ell \in (0,\eps]. 
\end{equation}
Let $\tau':= \inf\{t\ge \tau: X_t\neq X_\tau\}$. Observe that $B=\{\tau\le \eps<\tau'\}$, and thus
\begin{align*}
&\E_i\left[\int_0^\eps f(t,X_t,Q_{X_t}) dt\ \middle|\ B\right] \P(B)\\ %= \E\left[1_{\{\tau\le \eps<\tau'\}} \left(\int_0^{\tau} f(t,i,Q_{i}) dt+\int_\tau^\eps f(t,X_t,Q_{X_t})dt\right) \right]\\
&=\P(\tau\le \eps) \E\left[\left(\int_0^{\tau} f(t,i,Q_{i}) dt+\int_\tau^\eps f(t,X_t,Q_{X_t})dt\right) \P\left(\tau'>\eps\mid\tau \right) \middle|\ \tau\le \eps \right]\\
& = (1-e^{q_{ii}\eps}) \sum_{j\neq i} \frac{-q_{ij}}{q_{ii}} \int_0^\eps \left(\int_0^{\ell} f(t,i,Q_{i}) dt+\int_\ell^\eps f(t,j,Q_{j})dt \right) \eta(\ell) e^{q_{jj}(\eps-\ell)} d\ell\\
&= \sum_{j\neq i} q_{ij} \int_0^\eps \left(\int_0^{\ell} f(t,i,Q_{i}) dt+\int_\ell^\eps f(t,j,Q_{j})dt \right) e^{q_{ii}\ell} e^{q_{jj}(\eps-\ell)} d\ell,
\end{align*}
where the third line takes advantage of the fact
\begin{equation}\label{i to j}
\P(X_\tau=j \mid X_0=i) = {-q_{ij}}/{q_{ii}}\quad \forall j\neq i, 
\end{equation}
while the fourth line follows from \eqref{eta}. By estimates similar to \eqref{for Tt}, we get
\begin{align}
\E_i\bigg[\int_0^\eps f(t,X_t,Q_{X_t}) dt\ \bigg|\ B\bigg] \P(B)&= \sum_{j\neq i} q_{ij} \int_0^\eps \left( f(0,i,Q_{i})\ell + f(0,j,Q_{j})(\eps-\ell) \right)  d\ell + o(\eps^2)\notag\\
&= \frac12\bigg(-q_{ii} f(0,i,Q_i) + \sum_{j\neq i} q_{ij}  f(0,j,Q_j)\bigg) \eps^2 + o(\eps^2). \label{B part}
\end{align}
Thanks to \eqref{i to j}, we also have the estimate
\begin{align}
\P(C) &= \P(\tau< \eps, \tau'\in (\tau,\eps] \mid X_0=i) = \P(\tau<\eps \mid X_0=i) \sum_{j\neq i}\frac{-q_{ij}}{q_{ii}}\P(\tau'\in (\tau,\eps] \mid \tau<\eps, X_\tau=j)\notag\\
& \le (1-e^{q_{ii}\eps}) \sum_{j\neq i} \frac{-q_{ij}}{q_{ii}} (1-e^{q_{jj}\eps}) = O(\eps^2).\label{C part}
\end{align}
Now, using the fact that
\[
\P(X_\eps = j \mid X_0=i) = \left(e^{Q\eps}\right)_{ij} = \left(I + Q\eps + \frac12 Q^2\eps^2\right)_{ij} + o(\eps^2),\quad \forall i,j\in S,  
\]
a direct calculation shows 
\[
\E_i\left[\int_\eps^\infty f(t,X_t, Q^*_{X_t}) dt\right] = F_\eps(i,Q^*) + (Q_i\cdot F_\eps(Q^*))\eps + \frac12 ( (Q^2)_i\cdot F_\eps(Q^*))\eps^2 + o(\eps^2). 
\]
This, together with \eqref{A part}, \eqref{B part}, and \eqref{C part}, implies
\begin{align}
F(i,Q\otimes_\eps Q^*) =  F_\eps(i,Q^*) &+ (f(0,i,Q_i)+Q_i\cdot F_\eps(Q^*)) \eps \notag\\
&+ \frac12\left(  f_t(0,i,Q_i) +Q_i\cdot \vec f(0,Q) + (Q^2)_i \cdot F_\eps(Q^*) \right) \eps^2 +o(\eps^2),\label{2rd}
\end{align}
where $\vec f(0,Q) := (f(0,1,Q_1), f(0,2,Q_2),...,f(0,N,Q_N))$. 
Since $f_t$ satisfies \eqref{integrability}, $G(i,Q)$ is well-defined. Observe that
\begin{align*}
\lim_{\eps\downarrow 0}\left|\frac{F_\eps(i,Q)-(F(i,Q)+\eps G(i,Q))}{\eps}\right|\le \lim_{\eps\downarrow 0}\E_i\left[\int_0^\infty \frac{r(t,\eps;X_t,Q_{X_t})}{\eps} dt \right] =0,
\end{align*}
where the equality follows from \eqref{Taylor's} and the dominated convergence theorem, which is applicable here as $r(t,\eps;i,\bq)$ satisfies \eqref{h2} and \eqref{r1}. This shows that 
\begin{equation}\label{F_eps expan.}
F_\eps(i,Q)= F(i,Q) + \eps G(i,Q)+ o(\eps),
\end{equation}
and thus we can rewrite \eqref{2rd} as 
\begin{align}
F(i,Q\otimes_\eps Q^*) &= F_\eps(i,Q^*) + (f(0,i,Q_i)+Q_i\cdot F(Q^*)) \eps \notag\\
&+ \frac12\left( f_t(0,i,Q_i)+Q_i \cdot 2 G(Q^*)+ Q_i\cdot \vec f(0,Q)+ (Q^2)_i\cdot F(Q^*) \right) \eps^2 +o(\eps^2).\label{2rd'}
\end{align}
Observe from \eqref{Gamma} that
\[
Q\cdot \vec f(0,Q)^T + Q^2 \cdot F(Q^*)^T = Q\cdot \left(\vec f(0,Q)^T + Q \cdot F(Q^*)^T \right) = Q\cdot \Gamma^{Q^*}(Q)^T,
\]
where $v^T$ denotes the transpose of a vector $v\in\R^N$. This implies $Q_i\cdot\vec f(0,Q) + (Q^2)_i \cdot F(Q^*) = Q_i \cdot \Gamma^{Q^*}(Q)$. We therefore conclude that
\begin{align}
F(i,Q\otimes_\eps Q^*) &= F_\eps(i,Q^*) + \Gamma^{Q^*}(Q_i) \eps + \frac12\bigg( f_t(0,i,Q_i)+Q_i\cdot \left(2 G(Q^*) + \Gamma^{Q^*}(Q) \right) \bigg) \eps^2 +o(\eps^2).\label{2rd'}
\end{align}
By taking $Q= Q^*$ in \eqref{2rd'}, we get the corresponding expansion for $F(i,Q^*)$. Subtracting it from \eqref{2rd'} yields \eqref{2rd expan.}

%%%%%%%%%%%%%%%%%%%%%%%%

\subsection{Proof of Proposition~\ref{coro:sufficient for strong E}}\label{sec:sufficient for strong E}

%\begin{proof}
Fix $i\in S$. For any $Q\in \cQ$ with $Q_i \neq Q^*_i$, since \eqref{optimal ROC} holds with strict inequality, we observe from \eqref{o(1)} that $F(i,Q^*)-F(i,Q\otimes_{\eps} Q^*)> 0$ as $\eps>0$ is small enough. 

Now, take an arbitrary $Q\in \cQ\setminus \{Q^*\}$ such that $Q_i = Q^*_i$. Since $Q\neq Q^*$ but $Q_i=Q^*_i$, the collection 
$
S_0 := \{j\in S : Q_j\neq Q_j^*\}
$
must be nonempty. There are two distinct cases. 
\begin{itemize}[leftmargin=*]
%\item {\bf Case I:} $q_{ij} = q^*_{ij}=0$ for all $j\neq i$. Then $F(i,Q\otimes_{\eps} Q^*) = \int_0^\infty f(t,i,Q_i)dt =F(i,Q^*)$ for all $\eps>0$. 
\item {\bf Case I:} There is $\ell\in S_0$ such that $q^*_{i\ell}>0$. With $Q_i=Q^*_i$, we deduce from \eqref{2rd expan.} and \eqref{Lambda} that   
\begin{align}
\frac{F(i,Q^*)-F(i,Q\otimes_{\eps} Q^*)}{\eps^2} & = \frac12 Q^*_i\cdot \left(\Gamma^{Q^*}(Q^*) - \Gamma^{Q^*}(Q)\right) + o(1)\notag\\&
= \sum_{j\neq i} q^*_{ij} \left(\Gamma^{Q^*}(Q^*_j) - \Gamma^{Q^*}(Q_j)\right) + o(1). \label{simple 2rd}
\end{align}
Since \eqref{strict condition} entails $\Gamma^{Q^*}(Q^*_j) - \Gamma^{Q^*}(Q_j)> 0$ for $j\in S_0$ and $\Gamma^{Q^*}(Q^*_k) - \Gamma^{Q^*}(Q_k)= 0$ for $k\in S\setminus S_0$,
\[
\sum_{j\neq i} q^*_{ij} \left(\Gamma^{Q^*}(Q^*_j) - \Gamma^{Q^*}(Q_j)\right)\ge  q^*_{i\ell} \left(\Gamma^{Q^*}(Q^*_\ell) - \Gamma^{Q^*}(Q_\ell)\right)>0.
\]
It then follows from \eqref{simple 2rd} that $F(i,Q^*)-F(i,Q\otimes_{\eps} Q^*)>0$ as $\eps>0$ is small enough.
\item {\bf Case II:} $q^*_{ij}=0$ for all $j\in S_0$. Consider the stopping time
$
\tau := \inf\{t\ge 0: X_t\in S_0\}.
$
Note that $\tau>0$ as $i\notin S_0$, and that $Q_{X_t}=Q^*_{X_t}$ for all $t<\tau$. If there exists $\eps^*>0$ such that $\P(\tau\le \eps^*)=0$, then $F(i,Q\otimes_\eps Q^*) =F(i, Q^*)$ for all $0<\eps<\eps^*$. Therefore, we assume in the following that $\P(\tau\le \eps)>0$ for all $\eps>0$. For any $\eps>0$, 
\begin{align*}
F&(i,Q^*) - F(i,Q\otimes_\eps Q^*)\\
 &= \E_i\left[\int_0^\infty f(t,X_t,Q^*_{X_t}) dt - \int_0^\infty f(t,X_t,(Q\otimes_\eps Q^*)_{X_t}) dt\ \middle|\ \tau\le \eps \right] \P(\tau\le \eps)\\
&= \E_i\left[\int_\tau^\infty f(t,X_t,Q^*_{X_t}) dt - \int_\tau^\infty f(t,X_t,(Q\otimes_\eps Q^*)_{X_t}) dt\ \middle|\ \tau\le \eps \right] \P(\tau\le \eps)\\
&= \E_i\left[\int_0^\infty f(t+\tau,X^{X_\tau}_t,Q^*_{X^{X_\tau}_t}) dt - \int_0^\infty f(t+\tau,X^{X_\tau}_t,(Q\otimes_{\eps-\tau} Q^*)_{X^{X_\tau}_t}) dt\ \middle|\ \tau\le \eps \right] \P(\tau\le \eps)\\
&= \E_i\left[F_\tau(X_\tau,Q^*) - F_\tau(X_\tau, Q\otimes_{\eps-\tau}Q^*) \ \middle|\ \tau\le \eps \right] \P(\tau\le \eps),
\end{align*}
where $F_\tau$ is defined as in \eqref{F_eps} with $\eps$ therein replaced by $\tau$. 
%{\color{red}Since for each $X_\tau=j$
%$$\Gamma^{Q^*}(Q_j^*)-\Gamma^{Q^*}(Q_j)>0,$$
%by continuity of $\Gamma_y^{Q^*}(Q_j^*)$ and $\Gamma_y^{Q^*}(Q_j)$ there exists $y_j^*$ such that for any $y<y_j^*$
%$$|\Gamma_y^{Q^*}(Q_j^*)-\Gamma^{Q^*}_y(Q_j)\geq c>0.$$
%Then by the Lemma, there exists $\eps_j^*>0$ such that for any $\eps<\eps_j^*$ and $y\leq y_j^*\wedge a$,
%$$\frac{1}{\eps}[F_y(i,Q^*)-F_y(i,Q\otimes_\eps Q^*)]-[\Gamma_y^{Q^*}(i,Q^*)-\Gamma_y^{Q^*}(i,Q)]\geq -c/2.$$
%Then let $\eps^{**}:=\min_j(\eps_j^*\wedge y_j^*\wedge a)$, we have that for any $\eps<\eps^{**}$,
%$$F_\tau(X_\tau,Q^*) - F_\tau(X_\tau, Q\otimes_{\eps-\tau}Q^*)>0,$$
%which completes the proof.}

Consider the distribution function of $\tau$ conditioned on $\tau\le \eps$, i.e. $H^\eps(y):=\P(\tau\le y\mid \tau\le \eps)>0$, $y\in (0,\eps]$. The above equation can be rewritten as
\begin{align}
F&(i,Q^*) - F(i,Q\otimes_\eps Q^*)\notag\\
&= \P(\tau\le \eps) \sum_{j\in S_0} \P(X_\tau =j) \int_0^\eps \big(F_y(j,Q^*) - F_y(j, Q\otimes_{\eps-y}Q^*)\big) dH^\eps(y).\label{F-F}
\end{align}
By the proof of Lemma~\ref{lem:asymptotics'} in Appendix~\ref{sec:asymptotics'}, particularly \eqref{F_eps expan.}, we have
\begin{align}\label{F-F in int}
&F_y(j,Q^*) - F_y(j, Q\otimes_{\eps-y}Q^*)\notag\\
& = \big(F(j,Q^*) - F(j, Q\otimes_{\eps-y}Q^*)\big) + \big(G(j,Q^*) - G(j, Q\otimes_{\eps-y}Q^*)\big) y + o(y)\notag\\
&=  \left(\Gamma^{Q^*}(Q^*_{j})-\Gamma^{Q^*}(Q_{j})\right) (\eps-y)+\left(\bar\Gamma^{Q^*}(Q^*_{j})-\bar\Gamma^{Q^*}(Q_{j})\right) (\eps-y)y + o(\eps-y)+ o(y),
\end{align}
where $\bar\Gamma^{Q^*}(Q_{i}) : = f_t(0,i,Q_i)+Q_i\cdot G(Q^*)$. In the third line above, the first term follows from \eqref{expand F'} directly, while the second term is obtained by applying Lemma~\ref{lem:asymptotics} to $G$, in place of $F$. Observe from integration by parts that
\begin{align}
\int_0^\eps (\eps-y) dH^\eps(y) &= \eps - \int_0^\eps y dH^\eps(y) = \int_0^\eps H^\eps(y) dy.% = \eps H^\eps(y^*),\ \hbox{for some $0<y^*<\eps$}; 
\label{O(eps)}
%\int_0^\eps (\eps-y)y dH^\eps(y) &= \eps\bigg(\eps - \int_0^\eps H^\eps(y) dy\bigg) - \bigg(\eps^2-2\int_0^\eps H^\eps(y) y dy\bigg)\notag \\
%&=-\eps  \int_0^\eps H^\eps(y) dy +2\int_0^\eps H^\eps(y) y dy.\label{O(eps^2)}
\end{align}
To find a precise asymptotic expansion for $\int_0^\eps H^\eps(y) dy$, we need to analyze $H^\eps(y)$ further. Since $q^*_{ij}=0$ for all $j\in S_0$, to reach a state $j\in S_0$ from the current state $i$, at least two changes of states are needed. Specifically, consider $\tau_1:=\inf\{t\ge 0: X_t\neq i\}$ and $\tau_n:=\inf\{t\ge \tau_{n-1}:X_t\neq X_{\tau_{n-1}}\}$ for all $n\ge 2$. Let $\hat n := \inf\{n\in\N : \P_{i,Q^*}(X_{\tau_n}\in S_0)>0\}\ge 2$. 
%consider $\tau_1:=\inf\{t\ge 0: X_t\neq i\}$, $\tau_2:=\inf\{t\ge \tau_1:X_t\neq X_{\tau_1}\}$, and recall $\eta(\ell)$ in \eqref{eta}, the density function of $\tau_1$ conditioned on $\tau_1\le \eps$. Then,
By direct calculation,  
%\begin{align*}
%\P(\tau\le \eps) &= \P(\tau_1\le \eps) \sum_{j\in S\setminus S_0}\frac{q_{ij}}{-q_{ii}} \bigg(\int_0^\eps \P(\tau_2\le \eps-\ell\mid\tau_1=\ell) \eta(\ell) d\ell \bigg) \sum_{k\in S_0} \frac{q_{jk}}{-q_{jj}}+ o(\eps^2)\\
%&= \sum_{j\in S\setminus S_0} q_{ij} \sum_{k\in S_0} \frac{q_{jk}}{-q_{jj}} \int_0^\eps (1-e^{q_{jj} (\eps-\ell)}) e^{q_{ii}\ell} d\ell + o(\eps^2)\\
%&= \sum_{j\in S\setminus S_0} q_{ij} \sum_{k\in S_0} \frac{q_{jk}}{-q_{jj}} \int_0^\eps (-q_{jj} (\eps-\ell)) (1+q_{ii}\ell) d\ell + o(\eps^2)=  \bigg(\sum_{j\in S\setminus S_0,\ k\in S_0} q_{ij} q_{jk}\bigg)  \eps^2 + o(\eps^2)
%\end{align*}
%It follows that
%\begin{align*}
%H^\eps(y) = \frac{\P(\tau\le y)}{\P(\tau\le \eps)} = \frac{\big(\sum_{j\in S\setminus S_0,\ k\in S_0} q_{ij} q_{jk}\big)  y^2 + o(y^2)}{\big(\sum_{j\in S\setminus S_0,\ k\in S_0} q_{ij} q_{jk}\big)  \eps^2 + o(\eps^2)}
%\end{align*}
\begin{align}\label{K}
\P(\tau\le \eps) &=  K \eps^{\hat n} + o(\eps^{\hat n+1}),\quad \hbox{with}\quad  K:= \sum_{j_n\in S,\ j_{\hat n}\in S_0} q_{i,j_1} \cdot q_{j_1, j_2} \cdot ... \cdot q_{j_{\hat n-1}, j_{\hat n}} >0.
\end{align}
Note that $K$ is strictly positive by the definition of $\hat n$. We therefore obtain
\begin{align*}
H^\eps(y) = \frac{\P(\tau\le y)}{\P(\tau\le \eps)} = \frac{K  y^{\hat n} + o(y^{\hat n+1})}{K  \eps^{\hat n} + o(\eps^{\hat n+1})}.
\end{align*}
Since $H^\eps(y) \approx \frac{y^{\hat n}}{\eps^{\hat n}}$ as $\eps>0$ small, the leading-order term of $\int_0^\eps H^\eps(y) dy$ is $\int_0^\eps \frac{y^{\hat n}}{\eps^{\hat n}} dy = \frac{\eps}{\hat n+1}$. This, together with \eqref{O(eps)}, gives
\begin{equation}\label{O(eps)'}
\int_0^\eps (\eps-y) dH^\eps(y) = \int_0^\eps H^\eps(y) dy= \frac{\eps}{\hat n+1} + o(\eps),
\end{equation}
A calculation similar to \eqref{O(eps)} yields
\begin{equation}
\int_0^\eps (\eps-y)y dH^\eps(y) =-\eps  \int_0^\eps H^\eps(y) dy +2\int_0^\eps H^\eps(y) y dy = O(\eps^2). \label{O(eps^2)}
\end{equation}
Now, thanks to \eqref{F-F in int}, \eqref{K}, \eqref{O(eps)'}, and \eqref{O(eps^2)}, we deduce from \eqref{F-F} that
\begin{align}\label{F-F'}
F&(i,Q^*) - F(i,Q\otimes_\eps Q^*)\notag\\
&=\sum_{j\in S_0} \P(X_\tau =j) \left(\Gamma^{Q^*}(Q^*_j) - \Gamma(Q_j)\right) \frac{K}{\hat n+1} \eps^{\hat n+1}+ O(\eps^{\hat n+2}).
\end{align}
Since $\P(\tau\le \eps)>0$, there must exist $j\in S_0$ such that $\P(X_\tau =j)>0$. Also, recall that \eqref{strict condition} implies $\Gamma^{Q^*}(Q^*_j) - \Gamma(Q_j)>0$ for all $j\in S_0$. Hence, the constant in front of $\eps^{\hat n+1}$ in \eqref{F-F'} is strictly positive, which implies that $F(i,Q^*) - F(i,Q\otimes_\eps Q^*)>0$ as $\eps>0$ small enough. 
\end{itemize}
We have shown that for any $i\in S$ and $Q\in\cQ$, $F(i,Q^*) - F(i,Q\otimes_\eps Q^*)\ge 0$ as $\eps>0$ small enough. That is, $Q^*$ is a strong equilibrium. 
%\end{proof}

%%%%%%%%%%%%%%%%%%%%%%%%

\subsection{Proof of Theorem~\ref{thm:existence}}\label{sec:existence}

%\begin{proof}%[Proof of \thmref{t3}]
Define the set-valued map $\Phi: \cQ\to 2^{\cQ}$ by
\[
\Phi(Q) := \bigg\{R\in\cQ:\ R_i\in\argmax_{\bq\in D_i}\left[f(0,i,\bq)+F(Q)\cdot\bq\right],\quad \forall i\in S\bigg\}.
\]
For each $Q\in\cQ$, the compactness of $D_i$ and the continuity of the map $\bq\mapsto f(0,i,\bq)+F(Q)\cdot\bq$, for all $i\in S$, imply that $\Phi(Q)\neq\emptyset$. The same continuity also gives the closedness of $\Phi(Q)$. On the other hand, by the concavity of $\bq\mapsto f(0,i,\bq)$, $\Phi(Q)$ is convex.  

Next, we show that $\Phi$ is upper semicontinuous. Since $D_i$ is compact for all $i\in S$, $\cQ$ is also compact. The upper semicontinuity of $\Phi$ is then equivalent to the sequential characterization: for any $\{R^n\}_{n\in\N}$ and $\{Q^n\}_{n\in\N}$ in $\cQ$ with $R^n\to R$, $Q^n\to Q$, and $R^n\in\Phi(Q^n)$, we have $R\in\Phi(Q)$. To prove this, it suffices to show that the map 
\begin{equation}\label{needed conti.}
(\bq,Q)\mapsto f(0,i,\bq)+F(Q)\cdot\bq\quad \hbox{is continuous},\quad \hbox{for all $i\in S$}.
\end{equation}
Indeed, given $\bar R\in\cQ$, %by the closedness of $\cQ$, we can take $\{\bar R^n\}_{n\in\N}$ in $\cQ$ such that $\bar R^n \to \bar R$. Noting that 
$R^n\in\Phi(Q^n)$ implies  $f(0,i,R^n_i)+F(Q^n)\cdot R^n_i \ge f(0,i,\bar R_i)+F(Q^n)\cdot \bar R_i$ for all $i\in S$. We can then conclude from \eqref{needed conti.} that $f(0,i,R_i)+F(Q)\cdot R_i \ge f(0,i,\bar R_i)+F(Q)\cdot \bar R_i$ for all $i\in S$, as $n\to\infty$. Since $\bar R\in\cQ$ is arbitrarily chosen, this shows that $R\in\Phi(Q)$. 

Proving \eqref{needed conti.} boils down to establishing the continuity of 
\[
Q \mapsto F(Q) = \left(F(1,Q), F(2,Q), ... ,F(N,Q)\right). 
\]
%is continuous (w.r.t. $||\cdot||$). 
Take $\{Q^n\}_{n\in\N}$ in $\cQ$ such that $Q^n\rightarrow Q\in\cQ$. Denote by $\mu^n$ and $\mu$ the laws of $X$ under $Q^n$ and $Q$, respectively. Note that $\mu^n$ and $\mu$ are probability measures on $D([0,\infty);S)$, the space of c\`{a}dl\`{a}g processes taking values in $S$. %Because $X$ under $Q^n$ converges in finite-dimensional distribution to $X$ under $Q$, \red{we have $\mu^n$ converges weakly to $\mu$.}\footnote{\red{Don't we need any tightness here?}}. 
By \cite[p. 262, Problem 8]{EK-book-86}, $Q^n\rightarrow Q$ implies that $\mu^n$ converges weakly to $\mu$. 
Then, by the Skorokhod representaion theorem, there exists c\`{a}dl\`{a}g processes $Y^n$ and $Y$, defined on the same probability space $(\Omega,\mathcal{F},P)$, such that the laws of $Y^n$ and $Y$ are $\mu^n$ and $\mu$ respectively, and  
%$Y\overset{d}=\mu^n$, $Y^n\overset{d}=\mu$ and 
$Y^n\rightarrow Y$ under the Skorokhod topology on $D([0,\infty);S)$ $P$-a.s. 
In particular, we have $Y_t^n\rightarrow Y_t$, $P\times dt\text{-a.e}$. Since $S$ is a finite set, we in fact have $Y^n_t = Y_t$ for $n$ large enough, $P\times dt\text{-a.e}$. 
Then
$$F(i,Q^n)=\E^P\left[\int_0^\infty f(t,Y_t^n,Q_{Y_t^n}^n)dt\right]\rightarrow\E^P\left[\int_0^\infty f(t,Y_t,Q_{Y_t})dt\right]=F(i,Q).$$
This establishes the continuity of $Q \mapsto F(Q)$, and thus gives the upper semicontinuity of $\Phi$.

Now, we can apply Kakutani-Fan's fixed-point theorem (see e.g. \cite[Theorem 1]{Fan52}) to conclude that $\Phi$ admits a fixed point $Q^*\in\cQ$, i.e. $Q^*\in \Phi(Q^*)$. This implies that $Q^*$ satisfies \eqref{optimal ROC} for all $(i,Q)\in S\times\cQ$, and is thus a weak equilibrium, thanks to Theorem~\ref{thm:weak E iff}. 

If $f(0,i,\cdot)$ is strictly concave for all $i\in S$, %then $\Phi(Q)$ is a singleton for all $Q\in\cQ$. Thus, we have $\Phi(Q^*)=Q^*$, 
then the fixed point $Q^*$ satisfies strict inequalities $f(0,i,Q^*_i)+F(Q^*)\cdot Q^*_i > f(0,i, R_i)+F(Q^*)\cdot R_i$ for all $i\in S$ and $R\in\cQ$ with $R_i\neq Q^*_i$. That is, \eqref{strict condition} is satisfied. Thus, Proposition~\ref{coro:sufficient for strong E} asserts that $Q^*$ is a strong equilibrium.
%\end{proof}

%%%%%%%%%%%%%%%%%%%%%%%
%%%%%%%%%%%%%%%%%%%%%%%

\section{Proofs for Section~\ref{sec:discrete}}

\subsection{Proof of Theorem~\ref{thm:existence'}}\label{sec:existence'}
%\begin{proof}
Define the set-valued map $\Phi: \A\to 2^\A$ by
\[
\Phi(u):=\left\{w\in\A:\ w\in\argmax_{u'\in\A} V(i,u'\otimes_1 u),\ \forall i\in S\right\}.
\]
%We claim that $\Phi(u)$ is a nonempty closed convex set, for all $u\in\A$. 
Fix $u\in\A$. For any $i\in S$, define $g: \A_i\to\R$ by
\[
g(\alpha):=\kappa(0,i,\alpha)+\sum_{j=1}^N\left(\E_{j,u}\left[\sum_{t=0}^\infty \kappa(t+1,X_t,u_{X_t})\right]\cdot\alpha_j\right).
\]
Since $\kappa$ is continuous in $\alpha$, so is $g$. With $\mathfrak P$ being compact, $\A_x\subseteq \mathfrak P$ is also compact, and thus there exists a maximizer $\alpha(i)\in\A_i$ for $g$. By taking $w_i:=\alpha(i)$ for all $i\in S$, we get $w\in\Phi(u)$. For each $i\in S$, the continuity of $g:\A_i\to\R$ and the closedness of $\A_i$ also imply the closedness of the set of optimizers of $g$. It follows that $\Phi(u)$ is closed. Also, $g:\A_i\to\R$ is concave, thanks to the concavity of $\kappa$ in $\alpha$. The set of optimizers of $g$ is then convex, which yields the convexity of $\Phi(u)$.   

Next, we show that $\Phi$ is upper semicontinuous. That is, for any $u^n,u,w^n,w\in\A$ with $u^n\rightarrow u$, $w^n\rightarrow w$, if $w^n\in\Phi(u^n)$ then $w\in\Phi(u)$. It suffices to show that the map $(u',u)\mapsto V(i,u'\otimes_1 u)$ is continuous, for all $i\in S$. 
%Indeed, if $w^n\in\Phi(u^n)$ for all $n\in\N$, then for any $x\in S$ and $n\in\N$, $V(x,w^n\otimes u^n)\ge V(x,w'\otimes u^n)$ for all $w'\in\A$. As $n\to\infty$, we get $V(x,w\otimes u)\ge V(x,w'\otimes u)$ for all $w\in\A$ and $x\in S$. That is, $w\in \Phi(u)$. 
In view of \eqref{e1} and the concavity of $\kappa(0,i,\cdot)$, $u'\mapsto V(i,u'\otimes_1 u)$ is continuous. It remains to show that for any $j\in S$, the map
\[
h(u):=\E_{j,u}\left[\sum_{t=0}^\infty \kappa(t+1,X_t,u_{X_t})\right],\quad u\in\A
\]
is continuous. Take $\{u^n\}$ in $\A$ with $u^n\rightarrow u$. Fix $\eps>0$. By \eqref{e2}, there exists $T\in\N$ such that
$\sum_{t=T}^\infty\left(\max_{(i,\alpha)\in S\times\fP}\left|\kappa(t,i,\alpha)\right|\right)<\eps.$ 
It follows that  
\begin{eqnarray}
\notag&&\limsup_{n\rightarrow\infty}|h(u^n)-h(u)|\\
\notag&&\leq\limsup_{n\rightarrow\infty}\left|\E_{j,u}\left[\sum_{t=0}^T \kappa(t+1,X_t,u^n_{X_t})\right]-\E_{j,u}\left[\sum_{t=0}^T \kappa(t+1,X_t,u_{X_t})\right]\right|+2\eps\\
\notag&&=\limsup_{n\rightarrow\infty}\left|\sum_{t=0}^T\sum_{k=1}^N \kappa(t+1,k,u^n_k)((u^n)^t)_{jk}-\sum_{t=0}^T\sum_{z=1}^N \kappa(t+1,k,u_k)(u^t)_{jk}\right|+2\eps\\
\notag&&\leq\limsup_{n\rightarrow\infty}\sum_{t=0}^T\sum_{k=1}^N\left|\kappa(t+1,k,u^n(z))((u^n)^t)_{jk}-\kappa(t+1,k,u_k)(u^t)_{jk}\right|+2\eps=2\eps,%\\
%\notag&&= 
\end{eqnarray}
where the matrix $(u^n)^t$ (resp. $u^t$ ) is the $t$-fold product of $u^n$ (resp. $u$). By the arbitrariness of $\eps>0$, $h$ is continuous.

Now, by Kakutani-Fan's fixed-point theorem (see e.g. \cite[Theorem 1]{Fan52}), $\Phi$ admits a fixed point, i.e. there exists $u^*\in\A$ such that $u^*\in\Phi(u^*)$. That is, $u^*$ is an equilibrium.
%\end{proof}

%%%%%%%%%%%%%%%%%%%%%%%%%%%%%%%%%%%%%%%%%%%%%%%%%%%%%

\subsection{Proof of Lemma~\ref{l1}}\label{sec:l1}
%\begin{proof}
For clarity, in the proof below we will denote by $X$ the continuous-time Markov chain and $Y$ the discrete-time Markov chain, respectively. First, observe that
\begin{align*}
|&H_i^n(u^n)-F_i(Q^*)|\\
&\leq \left|\E_{i,u^n}\left[\sum_{k=0}^\infty f((k+1)\delta_n,Y_k,Q^n_{Y_k})\cdot\delta_n\right]-\E_{i,Q^n}\left[\sum_{k=0}^\infty f((k+1)\delta_n,X_{k\delta_n},Q^n_{X_{k\delta_n}})\cdot\delta_n\right]\right|\\
&\ \ \ \ +\left|\E_{i,Q^n}\left[\sum_{k=0}^\infty f((k+1)\delta_n,X_{k\delta_n},Q^n_{X_{k\delta_n}})\cdot\delta_n\right]-\E_{i,Q^n}\left[\int_0^\infty f(t,X_t,Q^n_{X_t})\,dt\right]\right|\\
&\ \ \ \ +\left|\E_{i,Q^n}\left[\int_0^\infty f(t,X_t,Q^n_{X_t})\,dt\right]-\E_{i,Q^*}\left[\int_0^\infty f(t,X_t,Q^*_{X_t})\,dt\right]\right|.
\end{align*}
Let $I_1^n$, $I_2^n$, and $I_3^n$ denote the second, third, and fourth line, respectively, in the above inequality.

Consider $I_1^n$ first. Let $\eps>0$. Recall $T>0$ given in \eqref{e33}. By \eqref{integrability}, there is $M>T$ such that
\begin{equation}\label{e34}
\int_M^\infty\sup_{i,\bq}|f(t,i,\bq)|\,dt<\eps.
\end{equation}
Then, by \eqref{e33}, we have 
\[
I_1^n\leq\delta_n\sum_{k<M/\delta_n}\left|\E_{i,u^n}\left[f((k+1)\delta_n,Y_k,Q^n_{Y_k})\right]-\E_{i,Q^n}\left[f((k+1)\delta_n,X_{k\delta_n},Q^n_{X_{k\delta_n}})\right]\right|+2\eps.
\]
Observe that
\begin{align*}
\E_{i,u^n}\left[f((k+1)\delta_n,Y_k,Q^n_{Y_k})\right]&=\sum_{j\in S}f((k+1)\delta_n,j,Q^n_{j})((u^n)^k)_{ij}\\
\E_{i,Q^n}\left[f((k+1)\delta_n,X_{k\delta_n},Q^n_{X_{k\delta_n}})\right]&=\sum_{j\in S}f((k+1)\delta_n,j,Q^n_j)((\tilde u^n)^k)_{ij},
\end{align*}
where $\tilde u^n$ is the transition matrix for the Markov chain $(X_{k\delta_n})_k$ induced by $Q_n$. That is, with the probability $\P$ induced by $Q^n=(q_{ij}^n)$, we have $\tilde u_{ij}^n=\P(X_{\delta_n}=j\,|\,X_0=i).$ Note that
$$\tilde u_{ij}^n=
\begin{cases}
1+q_{ii}^n\delta_n+o(\delta_n),& j=i,\\
q_{ij}^n\delta_n+o(\delta_n),& j\neq i,
\end{cases}\ 
=u_{ij}+o(\delta_n).
$$
It follows that $((\tilde u^n)^k)_{ij}=((u^n)^k)_{ij}+k\cdot o(\delta_n)\cdot(1+o(\delta_n))^k.$ Now, since
$$|k\cdot o(\delta_n)\cdot(1+o(\delta_n))^k|\leq \frac{M}{\delta_n}\cdot o(\delta_n)(1+\delta_n)^{M/\delta_n}=o(1),$$
we conclude that
$$I_1^n=\delta_n\sum_{i<M/\delta_n}o(1)+2\eps=o(1)+2\eps\rightarrow 2\eps.$$
By the arbitrariness of $\eps>0$, we get $I_1^n\rightarrow 0$.

Next, we consider $I_2^n$. For any $\eps>0$, recall $M>T$ given in \eqref{e34}. Then,
\[
I_2^n\leq\sum_{k<M/\delta_n}\E_{i,Q^n}\left|f\left((k+1)\delta_n,X_{k\delta_n},Q^n_{X_{k\delta_n}}\right)\delta_n-\int_{k\delta_n}^{(k+1)\delta_n}f\left(t,X_t,Q^n_{X_t}\right)\,dt\right|+2\eps.
\]
Since $f(\cdot,i,\cdot)$ is continuous, it is uniformly continuous on the compact set $[0,T]\times D_i^a$ with $D_i^a:=\{\bq\in D_i:\ ||\bq||\leq a\}$. As a result, there exists a modulus of continuity function $\rho$, independent of $i\in S$ and $n\in\N$, such that
$|f(t,i,Q^n_i)-f(t',i,Q^n_{i})|\leq\rho(|t-t'|).$
Consider
\[
J_k^n:=\left|f\left((k+1)\delta_n,X_{k\delta_n},Q^n_{X_{k\delta_n}}\right)\delta_n-\int_{k\delta_n}^{(k+1)\delta_n}f\left(t,X_t, Q^n_{X_t}\right)\,dt\right|,
\]
and the event
$A_k^n:=\{\text{there is no jump for $X$ in the time interval $(k\delta_n,(k+1)\delta_n]$}\}.$ 
Then, we have
\begin{align*}\E_{i,Q^n}[J_k^n]&=\E_{i,Q^n}[J_k^n|A_k^n]\cdot\P(A_k^n)+\E_{i,Q^n}[J_k^n|(A_k^n)^c]\cdot\P((A_k^n)^c)\\
&=(\rho(\delta_n)\cdot\delta_n)\cdot(1-O(\delta_n))+O(\delta_n)\cdot O(\delta_n)=o(\delta_n).
\end{align*}
As a result,
$$I_2^n\leq\sum_{k<M/\delta_n}o(\delta_n)+2\eps=o(1)+2\eps\rightarrow 2\eps.$$
By the arbitrariness of $\eps>0$, we get $I_2^n\rightarrow 0$.

Finally, from the argument in the proof of Theorem \ref{thm:existence}, we have $I_3^n\rightarrow 0$.
%\end{proof}

\bibliographystyle{siam}
\bibliography{refs}

\end{document}